\documentclass[12pt,reqno]{amsart}

\title[]{Lazy random walks and optimal transport on graphs}
\author{Christian L\'eonard}

\usepackage{amssymb, amsmath, amsfonts, latexsym, enumerate}
\usepackage[usenames,dvipsnames]{pstricks}
\usepackage{epsfig}
\RequirePackage[colorlinks,linkcolor=blue,citecolor=blue,urlcolor=blue]{hyperref}

\usepackage[english]{babel}
\usepackage[LGR,T1]{fontenc}
\usepackage[utf8]{inputenc}

\newtheorem{theorem}[equation]{Theorem}
\newtheorem{lemma}[equation]{Lemma}
\newtheorem{proposition}[equation]{Proposition}
\newtheorem{corollary}[equation]{Corollary}
\newtheorem{claim}[equation]{Claim}

\newtheorem{definition}[equation]{Definition}
\newtheorem{definitions}[equation]{Definitions}

\newtheorem{hypotheses}[equation]{Hypotheses}

\theoremstyle{remark}
\newtheorem{remark}[equation]{Remark}

\numberwithin{equation}{section}

\newcommand{\RR}{\mathbb{R}}

\newcommand{\1}{\mathbf{1}}

\newcommand\pf{_{\#}}

\newcommand{\as}{\textrm{-a.s.}}
\renewcommand{\ae}{\textrm{-a.e.}}

\newcommand{\scal}{\!\cdot\!}

    \DeclareMathOperator{\supp}{supp}

    \DeclareMathOperator{\argmin}{argmin}
    
	\DeclareMathOperator{\esssup}{-ess\, sup}

\newcommand{\boulette}[1]{$\bullet$\ Proof of #1.}
\newcommand{\Boulette}[1]{\par\medskip\noindent $\bullet$\ Proof of #1.}

\newcommand\seq[2]{(#1_#2)_{#2\ge1}}

\newcommand\Lim[1]{\lim_{#1\rightarrow\infty}}
\newcommand\Liminf[1]{\liminf_{#1\rightarrow\infty}}

\newcommand\Glim[1]{\Gamma\textrm{-}\lim_{#1\rightarrow\infty}}

\newcommand\lime{\lim_{\epsilon\rightarrow0}}

\newcommand\XX{\mathcal{X}}
\newcommand\XXX{\XX^2}
\newcommand\PX{\mathrm{P}(\XX)}
\newcommand\PdX{\mathrm{P}_2(\XX)}
\newcommand\PuX{\mathrm{P}_1(\XX)}
\newcommand\PXX{\mathrm{P}(\XXX)}
\newcommand\MX{\mathrm{M}_+(\XX)}
\newcommand\PY{\mathrm{P}(Y)}
\newcommand\PO{\mathrm{P}(\Omega)}
\newcommand\MO{\mathrm{M}_+(\Omega)}

\newcommand\OO{\Omega}
\newcommand\Oac{\Omega _{\textrm{ac,2}}}
\newcommand\OOt{\widetilde\Omega}

\newcommand\ii{{[0,1]}}
\newcommand\iX{{[0,1]\times\XX}}
\newcommand\IX{\int_{\XX}}
\newcommand\IXX{\int_{\XXX}}
\newcommand\IO{\int_\Omega}
\newcommand\Iii{\int_\ii}
\newcommand\IiX{\int_\iX}

\newcommand{\ZZ}{\mathcal{Z}}

\newcommand\Ph{\widehat{P}}
\newcommand\Rt{G}  
\newcommand{\Jty}{J ^{G,y}}
\newcommand\Jh{\widehat{J}}
\newcommand{\Ltxy}{L ^{G,xy}}
\newcommand{\Lty}{L ^{G,y}}

\newcommand\ph{\widehat{\pi}}
\newcommand{\JJ}{j}
\newcommand{\CC}{C _{\mathrm{kin}}}
\newcommand{\GG}{G}

\newcommand{\Gxy}{\Gamma ^{xy}}

\newcommand{\SMK}{\mathcal{S}_{ \mathrm{MK}}}

\newcommand{\sy}{\sum _{y:y\sim x}}
\newcommand{\sx}{\sum _{x\in\XX}}
\newcommand{\st}{\sum _{0< t< 1}}

\newcommand{\qq}{\mathbf{q}}
\newcommand{\vv}{\mathbf{v}}

\begin{document}


 \address{Modal-X. Universit\'e Paris Ouest. B\^at.\! G, 200 av. de la R\'epublique. 92001 Nanterre, France}
 \email{christian.leonard@u-paris10.fr}
\thanks{Author partially supported by the ANR project GeMeCoD. ANR 2011 BS01 007 01}

\begin{abstract} 
This paper is about the construction of  displacement interpolations on a discrete  metric graph. Our approach is based on the approximation of any optimal transport problem whose cost function is a distance on a discrete graph by a sequence of Schr\"odinger problems associated with random walks whose jump frequencies tend down to zero. Displacement interpolations are  defined as the limit of the time-marginal flows of the solutions to the Schr\"odinger problems.
This allows to work with these interpolations by doing stochastic calculus on the approximating random walks which are regular objects, and then to pass to the limit in a slowing  down procedure. 
The main convergence results are based on $\Gamma$-convergence of entropy minimization problems.
\\
 As a by-product, we obtain new results about  optimal transport on graphs. 
\end{abstract}

\maketitle 
\tableofcontents


\section*{Introduction}

Displacement interpolations on $\RR^n$ were introduced by McCann in \cite{McC94} and  extended later to a geodesic space $(\XX,d)$ where  they are defined as minimizing geodesics on the space of  all probability measures on $\XX$ equipped with the Wasserstein pseudo-distance of order two. They appeared to be a basic and essential notion of the Lott-Sturm-Villani theory  of lower bounded curvature of geodesic spaces, see \cite{LV09,St06a,St06b,Vill09}. Indeed, as discovered by McCann \cite{McC94,McC97}, Otto, Villani \cite{OV00}, Cordero-Erausquin, McCann, Schmuckenschl\"ager \cite{CMS01}, Sturm and von Renesse  \cite{SvR05} in the Riemannian setting, lower bounded curvature is intimately linked to convexity properties of the relative entropy with respect to the volume measure along  displacement interpolations. It happens that these convexity properties admit natural analogues on a geodesic space.

It is tempting to try to implement a similar approach in a discrete setting. But, 
little is known in this case since a discrete space  fails to be a length space. Indeed, any regular enough path on a discrete space is piecewise constant with instantaneous jumps,  so that no speed  and a fortiori no constant speed geodesic exist.

This paper is about the construction of  displacement interpolations on a discrete\footnote{The epithet \emph{discrete} is important since the standard definition of a (non-discrete) metric graph  allows for continuous mass transfer along the edges. In the present paper, only pure jumps occur.} metric graph. It permits us to propose, in this discrete setting,  natural substitutes for the constant speed geodesics.  As a by-product of our approach, we also obtain new results about the optimal transport on a graph.

To recover some time regularity of the paths without allowing any mass transfer along the edges of a graph, one is enforced to do some averaging on  ensembles of  discontinuous sample paths. This means that, for defining instantaneous speeds and accelerations, one is obliged to consider expected values of random walks. So doing, one lifts the  paths on the discrete state space $\XX$ up to  the continuous space $\PX$ of probability  measures on $\XX.$ 
This lifting from $\XX$ to $\PX$ has already been successfully used in the Lott-Sturm-Villani theory where minimizing geodesics on the length space $\XX$ are embedded in the set of all minimizing geodesics on the Wasserstein space of order two $(\PX,W_2)$, i.e.\ $W_2$-displacement interpolations. As usual, the pseudo-distance $W_2$ is defined as the square root of the optimal value of a transport problem with cost function  $d^2,$ the square of a distance $d$ on $\XX$.

The author already proposed in \cite{Leo12a} a construction of the $W_2$-displacement interpolations in $\RR^n$  as limits of solutions of Schr\"odinger problems when the reference processes are slowed down to a no-motion process. In the present paper, we stay as close as possible to this strategy.   It will lead us to a natural notion of $W_1$-displacement interpolation on a discrete metric graph.
 The main idea lies in the following thought experiment.

\emph{The cold gas  experiment}.
Suppose you observe at time $t=0$ a large collection of particles that are distributed with a profile close to the probability measure $\mu_0\in\PX$ on the state space $\XX.$ As in the thought  experiment proposed  by Schr\"odinger in 1931 \cite{Sch31,Sch32} or in its close variant  described in Villani's textbook \cite[\emph{Lazy gas experiment}, p.\,445]{Vill09}, ask them to rearrange into a new profile close to some $\mu_1\in\PX$ at some later time $t=1.$
\\
 Suppose that the particles are in contact with a heat bath. Since they are able to create mutual optimality (Gibbs conditioning principle), they  find an optimal transference plan between the endpoint profiles $\mu_0$ and $\mu_1.$ Now, suppose in addition that the  typical speed of these particles is close to zero: the particles are  lazy, or equivalently the heat bath is pretty cold. As each particle  decides to travel at the lowest possible cost,  it chooses an almost \emph{geodesic} path. Indeed, with a very high probability each particle is very slow, so that it is typically expected that its final position is close to its initial one. But it is required by the optimal transference plan  that it must reach a distant final position.  Hence, \emph{conditionally on the event that the target $\mu_1$ is finally attained}, each particle follows an almost geodesic path with a high probability. At the limit where the heat bath vanishes (zero temperature), each particle follows a geodesic while the whole system obeys the performs some optimal transference plan. This absolutely cold gas experiment is called the \emph{lazy gas experiment} in \cite{Vill09} where its dynamics is related to displacement interpolations. For further detail with graphical illustrations, see \cite[\S 6]{Leo12d}

With this thought experiment in mind,  one can guess that a slowing down procedure enforces  the appearance of individual geodesics.  A displacement interpolation is a  mixture of these geodesics which is specified by some optimal transport plan between the endpoint profiles $\mu_0$ and $\mu_1.$ We will see that displacement interpolations on a discrete metric graph $(\XX,\sim,d)$   are minimizing geodesics  on the Wasserstein space  of order one $(\PX,W_1)$ where the pseudo-distance $W_1$ on $\PX$ is defined as the optimal value of a transport problem with  the distance $d$  as its cost function.

\subsection*{Notation} Before going on we need some general notation.
We denote by $\PY$ and $\mathrm{M}_+(Y)$ the sets of all probability and positive measures on a measurable set $Y.$ The push-forward of a measure $ \alpha\in \mathrm{M}_+(Y_1)$ by the measurable mapping $f:Y_1\to Y_2$ is $f\pf \alpha (\cdot):= \alpha(f ^{ -1}(\cdot))\in \mathrm{M}_+(Y_2).$

Let $\OO\subset \XX ^{\ii}$  be a set of paths from the time interval $\ii$ to the measurable  state space $\XX.$  The canonical process  $X=(X_t)_{0\le t\le 1}$    
is defined for all   $\omega=(\omega_s)_{0\le s\le 1}\in \OO$ by $X_t(\omega)=\omega_t\in\XX$ for each $0\le t\le 1$. The set $\OO$ is endowed with the  $\sigma$-field generated by $(X_t;t\in\ii).$ 
For any $t\in\ii$ and any $Q\in \mathrm{M}_+(\OO),$ the push-forward 
\begin{equation}\label{eq-89}
Q_t:=(X_t)\pf Q\in\MX
\end{equation} 
of $Q$ by the measurable mapping $X_t$ is the law of the random position $X_t$ at time $t$ if $Q$ describes the   behaviour of the random path. More specifically $Q_0$, $Q_1$ are the initial and final time-marginal projections of $Q.$ Also
\begin{equation}\label{eq-80}
Q _{01}:=(X_0,X_1)\pf Q\in \mathrm{M}_+(\XXX)
\end{equation}
is the joint law of the random endpoint position $(X_0,X_1).$ The $xy$-bridge of $Q$ is the conditional probability measure
\begin{equation}\label{eq-81}
Q ^{xy}:=Q(\cdot|X_0=x,X_1=y)\in\PO,\quad x,y\in\XX.
\end{equation}
 As a general result, we have the disintegration formula
\begin{equation*}
Q(\cdot)=\IXX Q ^{xy}(\cdot)\, Q _{01}(dxdy)\in\MO.
\end{equation*}
If $P$ is a probability measure on $\OO,$ we sometimes use the probabilistic convention: $E_P(u):=\int _{\OO}u\, dP.$

\subsection*{Implementing the slowing down procedure}

As a consequence of the large deviation theory, 
the mathematical translation of the lazy gas experiment is in terms of some entropy minimization problems: the Schr\"odinger problems. More precisely, we are going to investigate the limit as a slowing down parameter $k$ tends to infinity (for instance, the average number of jumps of each individual particle is of order $1/k$) of the sequence of minimizing problems
\begin{equation}\label{eq-99}
H(P| R^k)\rightarrow \textrm{min};\qquad P \in\PO: P_0=\mu_0, P_1=\mu_1,
\end{equation}
where the relative entropy of a probability measure $p$ with respect to a reference measure $r$ is defined by $H(p|r):=\int \log(dp/dr)\,dp,$  $R^k\in\MO$  describes the random behaviour of a slow walker (a single typical particle) on the graph and $P$ is the unknown path probability measure which is subject to have initial and final prescribed marginal measures: $P_0=\mu_0$ and $P_1= \mu_1\in\PX,$ see notation \eqref{eq-89}.  For more detail about Schr\"o
dinger's problem, see the author's survey paper \cite{Leo12e}.
\\
By means of $\Gamma$-convergence, we  prove that the unique minimizer $\Ph^k$ of the entropy problem \eqref{eq-99} admits a limit 
$$
\Ph:=\lim_{k\to\infty}\Ph^k\in\PO
$$
that is characterized as a specific solution of a new auxiliary variational problem. It is the law of a random walk whose sample paths are piecewise constant geodesics with respect to a given distance $d$ on the graph $(\XX,\sim)$. Moreover, the joint law $\Ph_{01}\in\PXX$, see notation \eqref{eq-80}, of the couple of its endpoint positions  under $\Ph$ is a singled out solution of the Monge-Kantorovich optimal transport problem
\begin{equation}\label{eq-100}
\IXX d(x,y)\,\pi(dxdy)\rightarrow \textrm{min};\qquad\pi \in\PXX, \pi_0=\mu_0,\pi_1=\mu_1,
\end{equation}
where $\pi_0$ and $\pi_1\in\PX$ are the first and second marginals of $\pi.$ The optimal value of this problem is the Wasserstein distance of order one
$
W_1(\mu_0,\mu_1).$ 

This specific random walk $\Ph$ is a geodesic bridge between $\mu_0$ and $\mu_1.$ We  call it a \emph{displacement  random walk} and its time-marginal flow $(\Ph_t)_{0\le t\le1}\in\PX ^{\ii}$, see notation \eqref{eq-89}, defines a \emph{displacement interpolation} between $\mu_0$ and $\mu_1.$ This definition is justified because of several analogies with the standard displacement interpolations on a geodesic space.

\subsection*{Approximating \eqref{eq-100} by means of $ \Gamma$-asymptotic expansions}

Let us comment a little on the approximation of \eqref{eq-100} by \eqref{eq-99} as $k$ tends to infinity. Instead of the discrete set of vertices $\XX,$ let us first consider  the analogue of \eqref{eq-100} on $\XX=\RR^k$ which is related to Monge's problem: 
\begin{equation}\label{eq-101}
\int _{ \RR^k} d(x, T(x))\, \mu_0(dx)\rightarrow \textrm{min};\qquad T: \RR^k\to \RR^k, T\pf\mu_0=\mu_1,
\end{equation}
where the transport map $T$ is assumed to be measurable. The Monge-Kantorovich problem \eqref{eq-100} is a convex relaxation of \eqref{eq-101} in the sense that $\pi\mapsto \IXX d(x,y)\,\pi(dxdy)$ is a convex function on the convex subset $ \left\{\pi \in\PXX; \pi_0=\mu_0,\pi_1=\mu_1\right\} $ and $\pi^T:=(\mathrm{Id},T)\pf\mu_0$ gives $\IXX d(x,y)\,\pi^T(dxdy)=\IX d(x, T(x))\, \mu_0(dx).$

\subsubsection*{Many known  solutions of \eqref{eq-101}  rely on  approximations and variational methods}

Monge's original problem corresponds  to $d$ the standard Euclidean distance.  Sudakov   proposed an efficient, but still incomplete, strategy in \cite{Su79}.  The first complete solution was obtained by Evans and Gangbo in \cite{EG99}. It states that when $ \mu_0$ is absolutely continuous and $ \mu_0,\mu_1$ have finite first moments (plus some restrictions on $ \mu_0,\mu_1$), \eqref{eq-101} admits a unique solution. Its proof is based on PDE arguments and an approximation of the ``affine'' cost $d(x,y)=\|y-x\|$ by the ``strictly convex'' costs $d^ \epsilon(x,y):=\|y-x\| ^{ 1+\epsilon},$ with $ \epsilon>0$ tending to zero, which entails a convergence of the corresponding Monge-Kantorovich problems. A natural generalization of Monge's original problem is obtained by replacing the Euclidean norm by any  norm $\|\cdot\|$  on $\RR^k.$ With alternate approaches, but still taking advantage of the approximation $d ^{ \epsilon}\to d,$  
  Caffarelli
, Feldman and  McCann in \cite{CFM02} and Ambrosio and Pratelli   in \cite{Am03,AP03}, removed \cite{EG99}'s restrictions and extended this existence and uniqueness result to the case where the norm $\|\cdot\|$  is assumed to be strictly convex. Later, 
 Ambrosio, Kirchheim and Pratelli  \cite{AKP04} succeeded in the more difficult case where the norm is crystalline.  In the general case without any restriction on the norm, the solution has recently been obtained by Champion and De Pascale in \cite{CDP11}. Again, both \cite{AKP04} and \cite{CDP11} rely on variational methods and  $ \Gamma$- convergence. 
 
 The main $\Gamma$-convergence technic used during  the proofs of \cite{AP03,AKP04,CDP11} is an asymptotic expansion which was introduced by Anzellotti and Baldo \cite{AB93}, see also \cite{Att96}. In the present paper, we also make a crucial use of this technique at Lemma \ref{res-04}.  Instead of considering the approximation $d ^{ \epsilon}\to d,$ the not convex enough problem \eqref{eq-100} is approximated by the sequence of \emph{strictly} convex entropy minimization problems \eqref{eq-99} up to some normalization, see  \eqref{Skdyn} at the end of Section \ref{sec-analogy}.

\subsection*{A discrete metric measure graph $(\XX,\sim,d,m)$}

Let us take a distance $d$ on $\XX$ which is compatible with the graph structure $\sim$, in the sense of Hypothesis \ref{hyp-00}-($d$), and a possibly unbounded positive measure $m$ on $\XX.$
It will be shown that the sequence of reference Markov random walks $(R^k) _{k\ge1}$ can be chosen  such that for any $0\le t\le 1$ and any $k\ge1,$ the  $t$-marginal $R^k_t$ is equal to $m$, the Markov generator $L^k_t$ is self-adjoint in $L^2(\XX,m)$ and the limiting displacement interpolation 
$$
\mu_t:=\lim _{k\to \infty}\Ph_t^k\in\PX, \quad 0\le t\le1
$$ 
is a minimizing constant speed geodesic in $(\PX,W_1)$ where the Wasserstein distance $W_1$ is built upon the distance $d$. 

Although what follows is not treated in the present paper, it seems plausible that, in the perspective of deriving displacement convexity properties of the entropy, it is worthwhile studying the functions $t\in\ii\mapsto H\big((X_t)\pf\Ph^k|(X_t)\pf R^k\big)=H(\Ph^k_t|m)$ since the dynamics of $R^k$ and $\Ph^k$ are well understood, see  \cite{Leo12e}, and $\lim _{k\to \infty} H(\Ph^k_t|m)= H(\mu_t|m).$ This will be explored elsewhere.

\subsection*{Alternate approaches for deriving displacement convexity in a discrete setting}

Let us write a few words about already existing strategies in view of deriving  displacement convexity properties of the entropy on a discrete space.

\subsubsection*{Midpoint interpolations}

In the special case of the hypercube $\XX=\left\{0,1\right\}^ n$ equipped with the Hamming distance, Ollivier and Villani \cite{OV12} have introduced the most natural interpolation $(\mu_t)_{0\le t\le1}$  between $\delta_x$ and $\delta_y$ which is defined  as follows. For any $0\le t\le1,$ $\mu_t$ is the uniform probability measure  on the set of all $t$-midpoints of $x$ and $y.$ In this paper, the authors announce that it  seems to be difficult to obtain displacement convexity of the entropy along these interpolations. However, they prove a Brunn-Minkowski  inequality for the counting measure under the restriction  $t=1/2$ which allows for explicit combinatoric computations.

\subsubsection*{Approximate midpoint interpolations}

Bonciocat and Sturm \cite{BSt09} have introduced $h$-approximate midpoint interpolations and a natural extension of the lower bounded curvature of geodesic spaces, again in terms of displacement convexity of the relative entropy. They show that the  $h$-discretized graph of a geodesic space with a curvature  lower bound $K$ has a lower $h$-curvature bound which converges to $K$ as the discretization mesh $h$ tends to zero. They also compute $h$-curvature lower bounds of some planar graphs.

\subsubsection*{Maas-Mielke gradient flow}

Recently, Maas \cite{Maas11} and Mielke \cite{Mie11}  designed a new distance  $\mathcal{W}$ on $\PX$  which, unlike $W_2$, allows for regarding evolution equations of reversible Markov chains on the discrete space $\XX$ as gradient flows of the entropy $H(\cdot|m)$ on $(\PX, \mathcal{W}),$ where $m$ is the reversing measure of the Markov chain. Soon after, following  the Lott-Sturm-Villani strategy, Maas and  Erbar \cite{EM11}  and Mielke \cite{Mie13}    applied this gradient flow approach to obtain convexity properties of relative entropies in $(\PX,\mathcal{W})$ leading to new interesting results about lower bounded curvature on discrete spaces. 
\\
The distance $\mathcal{W}$ is a Riemannian distance   that is obtained by plugging discrete objects, such  as a discrete gradient and a discrete divergence, into the standard Benamou-Brenier formula, see \eqref{eq-77}. It is not clear that $\mathcal{W}$ and the corresponding minimizing geodesics on $(\PX, \mathcal{W})$ are associated  with an optimal transport problem related to some distance $d$ on $\XX$. Hence, although this approach is respectful of the measure space structure $(\XX,m)$, it might not be  linked to a \emph{metric}  structure $(\XX,d).$ This still needs to be clarified if, starting  from this Riemannian structure on $\PX$, one wishes to define a notion of Ricci type curvature on $\XX$ (which should be related to the variation of the discrete volume $m$ along ``geodesics'' on $(\XX,d)$). 

\subsubsection*{Entropic interpolations} Recently, the author \cite{Leo12d} studied convexity properties of the relative entropy along \emph{entropic} interpolations, i.e.\  time marginal flow of the solution of the entropy minimization problem \eqref{eq-99} without slowing down $(k=1).$

\subsubsection*{Binomial interpolations}

Consider the set $\left\{0,\dots,n\right\} $ with the graph structure of $\mathbb{Z}.$ The binomial interpolation between $\delta_0$ and $\delta_n$ is defined for any $0\le t\le1$ and $x\in \left\{0,\dots,n\right\} $ by $\mu_t(x)=\begin{pmatrix}
n\\x
\end{pmatrix}t^x (1-t)^{n-x}.$ It means that $\mu_t=\mathcal{B}(n,t)$ is the binomial distribution with parameters $n$ and $t$. This interpolation was introduced by Johnson \cite{Joh07} on $\mathbb{N}$ to obtain displacement convexity of the entropy with motivations slightly different from the Lott-Sturm-Villani theory. Now, on a graph equipped with the standard graph distance, a binomial interpolation between $\delta_x$ and $\delta_y$ is a mixture of the binomial interpolations along the geodesic chains $x=x_0\sim x_1\sim \cdots\sim x _{d(x,y)}=y$  that connect $x$ and $y$. These interpolations allow for displacement convexity of the  entropy with respect to the counting measure $m$. It has been successfully used to prove displacement convexity of the entropy  by Hillion \cite{Hil,Hil10,Hil12}   on some trees and  Gozlan, Roberto, Samson and Tetali \cite{GRST12} which have worked out the examples of the complete graph, $\mathbb{Z}^n$ and the hypercube. 
\\
It happens that binomial interpolations are specific instances of the displacement interpolations which are built in the present paper. They correspond to the simple random walk $R$ and the volume measure $m^o$, see \eqref{eq-98a} and \eqref{eq-98b} at the appendix, with the standard graph distance.
It will be seen at Claim \ref{res-45} that they are closely related to bridges of the Poisson process.

\subsection*{Outline of the paper}

Section \ref{sec-analogy} is a continuation of this introduction where we briefly present  the analogies between  usual displacement interpolations on a Riemannian manifold and displacement interpolations on a graph.
The  results are  stated at Sections \ref{sec-dRW} and \ref{sec-di}. Their proofs are done in the last Sections \ref{sec-convergence}, \ref{sec-dyn} and \ref{sec-conservation}.

Section \ref{sec-dRW} is devoted to the displacement random walks: Theorem \ref{res-0a} gives the  $\Gamma$-convergence results with in particular $\lim_{k\to\infty}\Ph^k=\Ph,$ and Theorems \ref{res-0b} and \ref{res-0c} describe the Markov dynamics of $\Ph.$
In Section \ref{sec-di}, the  results about displacement interpolations are stated. This section also  includes the proof of a Benamou-Brenier type formula at Theorem \ref{res-0e} and a discussion about constant speed interpolations and natural substitutes for the geodesics on a graph.  Theorem \ref{res-0f} is a  statement about the conservation of average rate of mass displacement. 
The Schr\"odinger problems are  introduced at Section \ref{sec-Sch} where a set of assumptions for the existence of their solutions is also discussed.
The $\Gamma$-convergence of the sequence of slowed down Schr\"odinger problems to the optimal transport problem of order one is studied  at Section \ref{sec-convergence}. The proofs rely on  Girsanov's formula for the Radon-Nykodim density $dR^k/dR.$
The dynamics of the limit $\Ph$ is worked out at Section \ref{sec-dyn}; some effort is needed to show that $\Ph$ is Markov. Finally, the conservation of the average mass displacement along interpolations is proved at  last Section \ref{sec-conservation}.

Basic information about random walks and  relative entropy with respect to an unbounded measure is provided at  the appendix sections A and B.

\section{Defining displacement interpolations on a graph by analogy}\label{sec-analogy}

To stress the  analogies between displacement interpolations in  discrete and continuous settings, we first  recall their main properties on a Riemannian manifold. Then, we briefly introduce the main properties of an object in a discrete setting, see \eqref{eq-76b}, which will be defined as a displacement interpolation because of the strong analogies between its properties and the corresponding properties of the displacement interpolations in a continuous setting, see Definitions \ref{def-03}.

\subsection*{McCann displacement interpolations}

The quadratic Monge-Kantorovich optimal transport problem on the metric space $(\XX,d)$ associated with the pair of prescribed probability measures $\mu_0,\mu_1\in\PX$ is 
\begin{equation}\label{eq-74}
\IXX \frac{1}{2} d^2(x,y)\,\pi(dxdy)\rightarrow \textrm{min};\qquad \pi\in\PXX: \pi_0=\mu_0,\pi_1=\mu_1
\tag{MK$_2$}
\end{equation}
where  we  denote $\pi_0(dx):=\pi(dx\times\XX)$ and $\pi_1(dy):=\pi(\XX\times dy)$ the marginals of $\pi\in\PXX.$ The Wasserstein pseudo-distance of order 2 is defined by
\begin{equation*}
W_2(\mu_0,\mu_1):=\sqrt{2\inf \eqref{eq-74}}
\end{equation*}
where $\inf \eqref{eq-74}$ is the value of the minimization problem \eqref{eq-74}. As usual, we denote 
\begin{equation*}
\PdX:=\left\{\mu\in\PX; \IX d^2(x_o,x)\,\mu(dx)<\infty\right\} 
\end{equation*}
for some $x_o\in\XX,$ so that $W_2$ is a distance on $\PdX.$
A displacement interpolation  is a  minimizing geodesic $[\mu_0,\mu_1]:=(\mu_t)_{t\in\ii}$ on $(\PdX,W_2)$ joining $\mu_0$ and $\mu_1$ in $\PdX$, i.e. 
\begin{equation*}
W_2(\mu_s,\mu_t)=|t-s|\ W_2(\mu_0,\mu_1),\quad s,t\in\ii.
\end{equation*}
On a Riemannian manifold $\XX$ equipped with its Riemannian distance $d$, it  appears that it is also an action minimizing geodesic in the following sense. Let  $\Oac$ be the space of all absolutely  continuous paths $\omega=(\omega_t)_{t\in\ii}$ from the time interval $\ii$ to $\XX$ such that $\Iii |\dot \omega_t|_{\omega_t}^2\,dt<\infty$ where $\dot \omega_t$ is the generalized derivative of $\omega$ at time $t$ and let $\mathrm{P}(\Oac)$ be the corresponding space of probability measures. It appears that
 $[\mu_0,\mu_1]$ is the time marginal flow (recall notation \eqref{eq-89})
\begin{equation}\label{eq-76}
\mu_t=\Ph_t,
\quad t\in\ii
\end{equation}
of some solution $\Ph\in\mathrm{P}(\Oac)$ of the following dynamical version of \eqref{eq-74},
\begin{equation}\label{eq-75}
\int _{\Oac}\CC(\omega) \,P(d \omega)\rightarrow \textrm{min};
\quad P\in \mathrm{P}(\Oac): P_0=\mu_0,P_1=\mu_1
\tag{MK$ _{\textrm{dyn},2}$}
\end{equation}
where the kinetic action $\CC$ is defined by
\begin{equation*}
\CC(\omega):=\Iii \frac{1}{2}|\dot \omega_t|_{\omega_t}^2\,dt\in[0,\infty],\quad \omega\in\Oac.
\end{equation*}
Suppose for simplicity that any solution $\pi^*\in\PXX$ of \eqref{eq-74} gives a zero mass  to the cut-locus so that there exists a unique minimizing geodesic $\gamma ^{xy}\in\Oac$ joining $x$ and $y$ for $\pi^*$-almost every $x,y\in\XX.$ 
Then, any solution $\Ph\in \mathrm{P}(\Oac)$ of \eqref{eq-75}  is in  one-one correspondence with a solution $\ph\in\PXX$ of \eqref{eq-74} via the relation 
\begin{equation}\label{eq-84}
\Ph(\cdot)=\IXX \delta _{\gamma ^{xy}}(\cdot)\, \ph(dxdy)\in \mathrm{P}(\Oac)
\end{equation}
where $\delta$ stands for a Dirac probability measure. 
With \eqref{eq-76}, we see that the displacement interpolation $[\mu_0,\mu_1]$ satisfies
\begin{equation}\label{eq-82}
\mu_t=\IXX \delta _{\gamma ^{xy}_t}(\cdot)\, \ph(dxdy)\in\PX,\quad t\in\ii.
\end{equation}
In particular, with $\mu_0=\delta_x$ and $\mu_1=\delta_y,$ we obtain $[\delta_x,\delta_y]=(\delta _{\gamma ^{xy}_t})_{t\in\ii}.$ This signifies that the notion of displacement interpolation lifts the notion of action minimizing geodesic from the manifold $\XX$ onto the Wasserstein space $\PdX$.

It follows from \eqref{eq-76} and \eqref{eq-75} that $W_2(\mu_0,\mu_1)$  admits the  Benamou-Brenier representation \cite{BB00}:
\begin{equation}\label{eq-77}
 W_2^2(\mu_0,\mu_1)=
	\inf _{(\nu,v)} \left\{\IiX  |v_t(x)|^2_x\, \nu_t(dx)dt\right\} 
\end{equation}
 where the infimum  is taken over all regular enough $(\nu,v)$ such that $\nu=(\nu_t)_{0\le t\le 1}\in \PdX ^{\ii},$ $v$ is a  vector field and these quantities are linked by the following current equation (in a weak sense) with boundary values:
\begin{equation*}
\left\{\begin{array}{l}
\partial_t \nu+\nabla\scal(\nu\, v)=0,\quad t\in (0,1)\\
\nu_0=\mu_0,\ \nu_1=\mu_1.
\end{array}\right.
\end{equation*}

\subsection*{Displacement interpolations on a discrete metric graph $(\XX,\sim,d)$}

The countable set $\XX$ of vertices is endowed with a graph structure: $x\sim y$ means that $x\ne y$ and $\left\{x,y\right\}$ is an undirected edge. Let $\OO\subset \XX ^{\ii}$ be the natural path space for a random walk on  the graph $(\XX,\sim):$ $\OO$ is the space of all left-limited, right-continuous,  piecewise constant paths $\omega=(\omega_t)_{0\le t\le1}$ on $\XX$ with finitely many jumps such that: $\forall t\in (0,1), \omega _{t^-}\ne \omega _{t}\Rightarrow \omega _{t^-}\sim \omega_t.$  The distance $d$ is in accordance with the graph structure, meaning that it is required to be intrinsic in the discrete sense, that  is
$ 
d(x,y)=\inf \left\{\ell(\omega); \omega\in\OO: \omega_0=x, \omega_1=y\right\} ,\ x,y\in\XX
$ 
with
\begin{equation}\label{eq-90}
\ell(\omega):=\st d(\omega _{t^-},\omega _{t}),\quad \omega\in\OO
\end{equation} 
the discrete length of the discontinuous path $\omega$. As before, we are going to define the displacement interpolation $[\mu_0,\mu_1]=(\mu_t)_{t\in\ii}$ by formula \eqref{eq-76}:
\begin{equation}\label{eq-76b}
\mu_t:=\Ph_t,\quad 0\le t\le 1
\end{equation}
 where $\Ph\in\PO$ is some \emph{singled out solution} of the following order-one analogue of \eqref{eq-75}:
\begin{equation}\label{MKdyn}
\IO \ell (\omega)\,P(d \omega)\to \textrm{min};\qquad P\in\PO: P_0=\mu_0, P_1=\mu_1.
\tag{MK$_{\textrm{dyn}}$}
\end{equation}
The random walk $\Ph$ minimizes the average length while transporting the mass distribution $\mu_0$ on $\XX$ onto another mass distribution $\mu_1.$\\ What is meant by ``singled out solution'' when talking about $\Ph$  will be made precise in a while. For the moment, let us say that $\Ph$ is selected among the infinitely many -- see \eqref{eq-91} below -- solutions of \eqref{MKdyn}, as being the limit  of a sequence of solutions of entropy minimizing problems associated with the slowing down procedure that was invoked when describing the lazy gas experiment. See \eqref{eq-93} below.

Since $d$ is assumed to be intrinsic, we shall see that the push-forward of \eqref{MKdyn}  onto $\XXX$ is the Monge-Kantorovich problem
\begin{equation}\label{MK}
\IXX d(x,y)\,\pi(dxdy)\rightarrow \textrm{min};\qquad \pi\in\PXX: \pi_0=\mu_0,\pi_1=\mu_1.
\tag{MK}
\end{equation}
It leads  to the Wasserstein pseudo-distance of order one
\begin{equation*}
W_1(\mu_0,\mu_1):=\inf \eqref{MK}
\end{equation*} 
which is a distance on
\begin{equation*}
\PuX:=\left\{\mu\in\PX; \IX d(x_o,x)\,\mu(dx)<\infty\right\}.
\end{equation*}
For each $x,y\in\XX,$ let us denote
\begin{equation*}
\Gamma ^{xy}:=\left\{\omega\in\OO; \omega_0=x, \omega_1=y, \ell(\omega)=d(x,y)\right\} 
\end{equation*}
 the set of all \emph{geodesics} joining $x$ and $y$.
 Remark that when $x$ and $y$ are distinct,  $\Gamma ^{xy}$ contains infinitely many paths since it is characterized by  ordered sequences of  visited states, regardless of the instants of jumps. On the other hand,
it is easily seen that for any measurable \emph{geodesic kernel}  $(Q ^{xy}\in \mathrm{P}(\Gamma ^{xy}); x,y\in\XX)$   and for any  $\pi^*$  solution of  \eqref{MK},  
\begin{equation}\label{eq-91}
P^*(\cdot):=\IXX Q ^{xy}(\cdot)\, \pi^*(dxdy)\in\PO
\end{equation}
solves \eqref{MKdyn}. It follows that  \eqref{MKdyn}  admits infinitely many solutions and also that the static and dynamical Monge-Kantorovich problems have the same optimal value:
\begin{equation}\label{eq-97}
\inf \eqref{MK}=\inf \eqref{MKdyn}.
\end{equation}

The $(X_t)_{t\in\ii}$-push forward of the minimizer $P^*$ given at \eqref{eq-91} is 
\begin{equation}\label{eq-92}
P^*_t(\cdot)=\IXX Q ^{xy}_t(\cdot)\, \pi^*(dxdy)\in\PX,
\quad 0\le t\le1.
\end{equation}
Remark that \eqref{eq-91}-\eqref{eq-92} has  the same structure as \eqref{eq-84}-\eqref{eq-82}.
The slowing down procedure  selects one singled out geodesic kernel $(\GG ^{xy}\in \mathrm{P}(\Gamma ^{xy});x,y\in\XX)$ which does not depend on the specific choice of $\mu_0$ and $\mu_1$, see Theorem \ref{res-0a} below, and one singled out solution $\ph\in\PXX$  of \eqref{MK} such that 
\begin{equation}\label{eq-86}
\Ph(\cdot)=\IXX \GG ^{xy}(\cdot)\,\ph(dxdy)\in\PO
\end{equation}
and the displacement interpolation $[\mu_0,\mu_1]$  satisfies
\begin{equation}\label{eq-83}
\mu_t(\cdot)=\IXX \GG ^{xy}_t(\cdot)\,\ph(dxdy)\in\PX,\quad t\in\ii
\end{equation}
where for each $x,y\in\XX$, $\GG_t ^{xy}\in\PX$ is the $t$-marginal of $\GG ^{xy}\in \mathrm{P}(\Gamma ^{xy}).$

It will be proved that $\Ph$ and every $G ^{xy}$ have the Markov property.

The defining identity $\mu_t:=\Ph_t,$ $t\in\ii,$ states that $\Ph$ is a dynamical coupling of $[\mu_0,\mu_1].$

Comparing \eqref{eq-84} and \eqref{eq-86} leads us to the following analogies:
\begin{itemize}
\item
The optimal plan $\ph$ in \eqref{eq-83} refers to \eqref{MK}, while in \eqref{eq-82} it refers to \eqref{eq-74}.
\item
The Markov random walk $\GG ^{xy}\in \mathrm{P}(\Gamma ^{xy})$ in \eqref{eq-83} corresponds to  $\delta _{\gamma ^{xy}}$ in \eqref{eq-82}. In particular, we see that the deterministic behaviour of $\delta _{\gamma ^{xy}}$ must be replaced with a genuinely random walk $\GG ^{xy}.$ 
\end{itemize}
\emph{
The  geodesic kernel $(\GG ^{xy}\in \mathrm{P}(\Gamma ^{xy});x,y\in\XX)$ encodes some geodesic dynamics of the discrete metric graph $(\XX,\sim,d).$}

The analogy between \eqref{eq-84}-\eqref{eq-82} and \eqref{eq-86}-\eqref{eq-83} entitles us to propose the following definitions.

\begin{definitions}\label{def-03}
We call $\mu=(\Ph_t)_{0\le t\le1}$ the  $(R,d)$-displacement interpolation and $\Ph$ the $(R,d)$-displacement random walk between $\mu_0$ and $\mu_1.$
\\
We denote $\mu=[\mu_0,\mu_1]^{(R,d)}$ or more simply $[\mu_0,\mu_1]^{R}$ or $[\mu_0,\mu_1]$ when the context is clear.
\end{definitions}

It follows from $\mu_t:=\Ph_t, 0\le t\le 1,$  \eqref{MKdyn} and the Markov property of $\Ph$ that $W_1(\mu_0,\mu_1)$  admits the  Benamou-Brenier type representation
\begin{equation}\label{eq-59b}
W_1(\mu_0,\mu_1)
= \inf _{(\nu,\JJ)}\Iii dt \IXX d(z,w)\,\nu _{t}(dz) \JJ _{t,z}(dw)<\infty 
\end{equation}
where the infimum is taken over all  couples $(\nu,\JJ)$ such that $\nu=(\nu_t)_{t\in\ii}\in \PX ^{\ii}$ is a time-differentiable  flow of probability measures on $\XX,$ $\JJ=(\JJ _{t,z})_{t\in\ii,z\in\XX}\in \MX ^{\ii\times\XX}$ is a measurable jump kernel, $\nu$ and $\JJ$ are linked by the current equation
\begin{equation*}
\left\{
\begin{array}{l}
\partial_t \nu_t(z)+\IX [\nu_t(z)\JJ _{t,z}(dw)-\nu_t(dw)\JJ _{t,w}(z)]=0,\qquad 0<t<1,\ z\in\XX\\
\nu_0=\mu_0,\ \nu_1=\mu_1
\end{array}
\right.
\end{equation*}
with $\IX \nu_t(z)\JJ _{t,z}(dw)<\infty$ for all $0<t<1,\ z\in\XX.$
The infimum
 $\inf _{(\nu,\JJ)}$  is attained at $(\mu,\Jh)$ where $\mu$ is  the displacement interpolation and $\Jh$ is the jump kernel of the Markov displacement random walk $\Ph$. Hence
\begin{equation*}
W_1(\mu_0,\mu_1)= \Iii dt \IXX d(z,w)\,\mu_t(dz)\Jh _{t,z}(dw).
\end{equation*}
This is in complete analogy with the standard Benamou-Brenier formula \eqref{eq-77}. 
The detailed formulation of this Benamou-Brenier formula is given at Theorem \ref{res-0e}.

\subsection*{The discrete  metric measure graph $(\XX,\sim, d,m)$}

Let us equip $(\XX,\sim,d)$ with a positive measure $m\in\MX.$
 We introduce a sequence $R_k\in\MO$ of \emph{slowed down} continuous-time Markov random walks on $\XX$ which is respectful of the graph  and the metric measure structures of $(\XX,\sim, d,m)$.

For each $k\ge1,$ the Markov generator $L^k=(L^k _{t})_{0\le t\le1}$ ot $R^k$
is defined  for any finitely supported function $u$ by
\begin{equation}\label{eq-85}
L^k_tu(x)=\sum _{y:y\sim x}k ^{-d(x,y)} J _{t,x} (y) [u(y)-u(x)],\quad  t\in\ii, x\in\XX, k\ge1.
\end{equation}
In this formula $J _{t,x}(y)>0$ is the average rate of jumps of the random walk from $x$ to $y$ at time $t$ when $k=1$. For some detail about random walks, see the appendix Section \ref{sec-RW}.
\\
 As $k$ increases, the sequence of generators $(L^k)_{k\ge1}$
tends down to zero. It  describes the dynamics of slowed down random walks which ultimately do not move anymore. We call them \emph{lazy random walks}.
For each $k$, the random walk with generator $L^k$ and initial measure $m\in\MX$ is described by a positive measure $$R^k\in\MO$$ on the path space $\OO$ (the letter $R$ stands for \emph{reference} measure). 

We also assume that the graph $(\XX,\sim)$ is irreducible and that the initial positive measure $m=(m_x)_{x\in\XX}$ and the jump kernel $J$ satisfy the detailed balance conditions
\begin{equation}\label{eq-94}
m_xJ _{t,x}(y)=m_y J _{t,y}(x),\quad \forall x,y\in\XX, t\in\ii.
\end{equation}
Together with the irreducibility assumption, these conditions imply that $m_x>0,\forall x\in\XX,$ whenever $m(\XX)>0,$ which is always assumed. 
A large class of $m$-reversible random walks is described at the appendix Section \ref{sec-RW}, see \eqref{eq-44}.

Let us comment on the connection between this sequence of lazy random walks and the metric graph measure structure of $(\XX,\sim,d,m).$ 
\begin{enumerate}
\item[($\sim$)]
As  for each $k, t$ and $x$, the support of the jump measure $J _{t,x}=\sum _{y:y\sim x}J _{t,x}(y)\, \delta_y\in\MX$ is the set of neighbours of $x$,  the random walker which is governed by $R^k$ is allowed to jump  from $x$ to $y$ if and only if $x$ and $y$ are neighbours.

\item[($m$)]
Because of \eqref{eq-94}, the Markov random walks $R^k$ are  assumed to be $m$-stationary, i.e.\ for all $t\in\ii$ the $t$-marginal $R^k_t$ of $R^k$ satisfies $R^k_t=m$. Furthermore, for all $0\le t\le1$ and $k\ge1,$ the generators $L_t^k$ are self-adjoint on $L^2(\XX,m).$
\item[($d$)]
The connection with the distance $d$ is less immediate.  For any $x,y\in\XX,$ we shall prove at Theorem \ref{res-11} that the sequence of bridges $R ^{k,xy}\in\PO$ from $x$ to $y$ (see notation \eqref{eq-81}) converges to a probability measure which is concentrated on  the set $\Gamma ^{xy}$ of all geodesics joining $x$ and $y.$
\end{enumerate}
Therefore, in some sense, the sequence $(R^k)_{k\ge1}$ of lazy random walks  is respectful of the discrete  metric measure graph structure of $(\XX,\sim, d,m)$.

\subsection*{Optimal transport appears at the limit of the slowing down procedure}

Let us briefly sketch the connection between lazy random walks and optimal transport. The main idea is to consider, as in \cite{Leo12a}, a sequence of entropy minimizing problems which are called Schr\"odinger problems. This is the mathematical implementation of the lazy gas experiment that was described in the introduction. Details about the connection between the lazy gas experiment and the Schr\"odinger problem are given in the author's survey paper \cite[\S 6]{Leo12e}. 

Recall that the relative entropy of a probability measure $p$ on the measurable space $Y$ with respect to some reference positive measure $r\in \mathrm{M}_+(Y)$ on $Y$ is roughly defined by 
\begin{equation*}
H(p|r):=\int_Y \log(dp/dr)\, dp\in(-\infty,\infty],\quad p\in \mathrm{P}(Y)
\end{equation*}
if $p$ is absolutely continuous with respect to $r$ and $+\infty$ otherwise. For a rigorous definition, see the appendix Section \ref{sec-ent}. For each $k\ge2,$ consider the entropy minimizing problem that is called a dynamical \emph{Schr\"odinger problem} 
\begin{equation}\label{Skdyn}
H(P| R^k)/\log k\rightarrow \textrm{min};\qquad P\in\PO: P_0=\mu_0, P_1=\mu_1.
 \tag{S$_{\mathrm{dyn}}^{k}$}
\end{equation}
See \cite{Leo12e} for more detail about  Schr\"odinger's problem.
It should be compared with \eqref{MKdyn}. Note that unlike the affine minimization problem \eqref{MKdyn}, \eqref{Skdyn} is a \emph{strictly} convex minimization problem. Hence, \eqref{Skdyn} admits at most one solution while \eqref{MKdyn}  admits an infinite convex set of solutions.
\\
It is proved at Lemma \ref{res-03} that
\begin{equation*}
\Glim k \eqref{Skdyn}=\eqref{MKdyn}
\end{equation*}
where this $\Gamma$-convergence refers to the standard narrow topology $\sigma(\PO,C_b(\OO))$ on $\PO$ when $\OO$ is endowed with the Skorokhod topology, see \eqref{eq-19} for  detail about $\OO.$ For each $k$, let $\Ph^k\in\PO$ denote the \emph{unique} solution of \eqref{Skdyn}. As a  consequence of this $\Gamma$-limit, one expects that any  limit point of $(\Ph^k)_{k\ge1}$ solves \eqref{MKdyn}. We shall do better at Theorem \ref{res-0a} which states that $(\Ph^k)_{k\ge1}$ is a convergent sequence and that 
\begin{equation}\label{eq-93}
\Lim k\Ph^k:=\Ph\in\PO
\end{equation}
is represented by \eqref{eq-86} where the geodesic kernel $(\GG ^{xy})_{x,y}$ is expressed in terms of the geodesic sets $(\Gamma ^{xy})_{x,y}$, $J$ and $R:=R ^{k=1}.$ 
\\
Pushing forward \eqref{Skdyn} from $\PO$ onto $\PXX$ via the $(0,1)$-marginal projection $(X_0,X_1)$ gives us 
\begin{equation}\label{Sk}
H(\pi| R _{01}^k)/\log k\to \textrm{min};\qquad \pi\in\PXX: \pi_0=\mu_0,\ \pi_1=\mu_1
 \tag{S$^k$}
\end{equation}
where $R _{01}^k\in\PXX$ is  the joint law of the initial and final positions of the random walk $R^k$, see notation \eqref{eq-80}. A consequence of $\Glim k \eqref{Skdyn}=\eqref{MKdyn}$ is
\begin{equation*}
\Glim k \eqref{Sk}=\eqref{MK}
\end{equation*}
and it will proved that for each $k,$ the unique solution $\ph^k\in\PXX$ of \eqref{Sk} is the joint law   of the initial and final positions of the random walk $\Ph^k,$ i.e.\ $\ph^k:=\Ph^k _{01}$.   
Therefore, the slowing down procedure $\Lim k\Ph^k=\Ph$ selects
\begin{enumerate}
\item
one solution $\ph:=\Lim k\ph^k\in\PXX$ of the static problem  \eqref{MK} and
\item
one random dynamics encoded in $(\GG ^{xy};x,y\in\XX).$
\end{enumerate}
It is interesting to note that the convergence $\Lim k\ph^k\in\PXX=\ph$ is a by-product of its dynamical analogue.

\section{Main results about  displacement random walks}\label{sec-dRW}

We gather our assumptions before stating our main results about the displacement random walks.

\subsection*{The underlying hypotheses}

The following set of hypotheses will prevail for  the rest of the paper.

\begin{hypotheses}\label{hyp-00}
The vertex set $\XX$ is countable.
\begin{enumerate}
\item[($\sim$)]
\begin{enumerate}[-]
\item
$(\XX,\sim)$ is irreducible: for any $x,y\in\XX,$ there exists a finite chain $x_1,x_2,\dots,x_n$ in $\XX$ such that $x=x_1\sim x_2\sim\cdots\sim x_n=y.$
\item
$(\XX,\sim)$ contains no loop: $x\sim x$ is forbidden.
\item
$(\XX,\sim)$ is locally finite:  any  vertex $x\in\XX$ admits finitely many neighbours
\begin{equation}\label{eq-03}
n_x:= \# \left\{y\in\XX; y\sim x\right\} <\infty,\quad \forall x\in\XX.
\end{equation}
\end{enumerate}
\item[($d$)]
\begin{enumerate}[-]
\item
The distance $d$ is positively lower bounded: for all $x\ne y\in\XX,$ $d(x,y)\ge1.$
\item
The distance $d$ is intrinsic in the discrete sense: 
\begin{equation*}
d(x,y)=\inf \left\{\ell(\omega); \omega\in\OO: \omega_0=x, \omega_1=y\right\} ,\quad x,y\in\XX
\end{equation*}
where the discrete length $\ell$ is defined at \eqref{eq-90}.
\end{enumerate}

\item[($R$)]
The reference path  measure $R\in\MO$ is assumed to be Markov with a forward  jump kernel $(J _{t,x}\in\MX; t\in\ii,x\in\XX)$ such that:
\begin{enumerate}[-]
\item 
For any $x,y\in\XX,$ we have
$J _{t,x}(y)>0,$ $\forall t\in\ii$ if and only if $x\sim y$.
\item
$J$ is uniformly bounded, i.e.
\begin{equation}\label{eq-01b}
\sup _{t\in\ii,x\in\XX}J _{t,x}(\XX)<\infty.
\end{equation}
\end{enumerate} 

\item[($R^k$)]
For each $k\ge1,$ the slowed down random walk $R^k\in\MO$ is the Markov measure with the forward jump kernel
\begin{equation}\label{eq-06b}
J^k _{t,x}:=\sy k ^{-d(x,y)}J _{t,x}(y)\,\delta_y,\qquad t\in\ii, x\in\XX, 
\end{equation}
and the initial measure is $R^k_0=m\in\MX$ with $m_x>0$ for all $x\in\XX.$
\item[($\mu$)]
The prescribed probability measures $\mu_0$ and $\mu_1\in\PX$ satisfy the following requirements. 
There exists   some $\pi^o\in\PXX$ such that\\ $\pi^o_0=\mu_0,$ $\pi^o_1=\mu_1$, $\IXX E _{R ^{xy}}(\ell)\, \pi^o(dxdy)<\infty$ and $H(\pi^o|R _{01})<\infty$.

\end{enumerate}
\end{hypotheses}

We give a simple criterion for the Hypothesis ($\mu$) to be verified.
Remark that for the problems $\eqref{Skdyn}$ and $\eqref{Sk}$ to admit solutions, it is necessary that $H(\mu_0|R_0), H(\mu_1|R_1)<\infty$.

\begin{proposition}\label{res-43}

For the Hypothesis \ref{hyp-00}-($\mu$) to be satisfied, it is enough that in addition to $H(\mu_0|R_0), H(\mu_1|R_1)<\infty$,  there exists a nonnegative  function $A$ on $\XX$ such that
\begin{enumerate}[(i)]
\item
$\IXX e ^{-A(x)-A(y)}R _{01}(dxdy)<\infty$
\item
$R _{01}(dxdy)\ge e ^{-A(x)-A(y)}\, R_0(dx)R_1(dy)$
\item
$E _{R ^{xy}}(\ell)\le A(x)+A(y)$, for all $x,y\in\XX$
\item
$\IX A\,d \mu_0,\IX A\, d \mu_1<\infty.$
\end{enumerate}
\end{proposition}

In particular, Hypothesis \ref{hyp-00}-($\mu$) holds when $\XX$ is finite.

The proof of Proposition \ref{res-43} is done at Section \ref{sec-Sch}.

\subsection*{The minimization problems}

Let us recall for convenience the Monge-Kantorovich problems
\begin{equation*}
\IXX d(x,y)\,\pi(dxdy)\rightarrow \textrm{min};\qquad \pi\in\PXX: \pi_0=\mu_0,\pi_1=\mu_1
\tag{MK}
\end{equation*}
and
\begin{equation*}
\IO \ell (\omega)\,P(d \omega)\to \textrm{min};\qquad P\in\PO: P_0=\mu_0, P_1=\mu_1
\tag{MK$_{\textrm{dyn}}$}
\end{equation*}
and the Schr\"odinger problems which are defined for each $k\ge1,$ by
\begin{equation*}
H(\pi| R _{01}^k)/\log k\to \textrm{min};\qquad \pi\in\PXX: \pi_0=\mu_0,\ \pi_1=\mu_1
 \tag{S$^k$}
\end{equation*}
and
\begin{equation*}
H(P| R^k)/\log k\rightarrow \textrm{min};\qquad P\in\PO: P_0=\mu_0, P_1=\mu_1.
 \tag{S$_{\mathrm{dyn}}^{k}$}
\end{equation*}
We denote $$\Gamma:=\cup _{x,y\in\XX}\Gxy$$ the set of all geodesics. It is proved at Lemma  \ref{res-meas} that it is measurable, so that one is allowed to define the path measure
\begin{equation}\label{eq-35}
\Rt:= \1_\Gamma\exp \left(\Iii J _{t,X_t}(\XX)\,dt\right) \,R\in\MO.
\end{equation}
Let us denote $\SMK(\mu_0,\mu_1)\subset\PXX$ the set of all solutions to the Monge-Kantorovich problem \eqref{MK} and introduce the subsequent auxiliary entropic minimization problems
\begin{equation}\label{St}
H(\pi|\Rt _{01})\to \mathrm{min};\qquad \pi\in \SMK(\mu_0,\mu_1).
\tag{$\widetilde{\mathrm{S}}$}
\end{equation}
and
\begin{equation}\label{Stdyn}
H(P|\Rt)\to \mathrm{min};\qquad P\in\PO: P _{01}\in \SMK(\mu_0,\mu_1).
\tag{$\widetilde{\mathrm{S}}_{\mathrm{dyn}}$}
\end{equation}

\subsection*{Results about the displacement random walks}
We are now ready to state the main results about the random walks. Their consequences in terms of interpolations will be made precise at next section.

\begin{theorem}\label{res-0a}
The Hypotheses \ref{hyp-00} are assumed to hold.

\begin{enumerate}
\item
For all $k\ge2$, the problems \eqref{Sk} and \eqref{Skdyn}  admit respectively a unique solution $\ph^k\in\PXX$ and $\Ph^k\in\PO.$ \\ Moreover, $\Ph^k$ is Markov and  $\ph^k=\Ph^k _{01}.$
\item
\eqref{St} has a unique solution $\ph\in\PXX$ and  $\Lim k \ph^k=\ph.$
\\ As a definition, $\ph$ also solves \eqref{MK}.
\item
 \eqref{Stdyn} has a unique solution $\Ph\in\PO$ and $\Lim k \Ph^k=\Ph.$ 
\\ The limit $\Ph$ also solves \eqref{MKdyn}.
\item
$\Ph$ is the following mixture of bridges of $\Rt$:
\begin{equation}\label{eq-22}
\Ph(\cdot)=\IXX \Rt ^{xy}(\cdot)\,\ph(dxdy)\in\PO,
\end{equation}
meaning that it  satisfies $\Ph _{01}=\ph$ and that $\Ph$ shares its bridges with the geodesic path measure $G$ defined by \eqref{eq-35}: $\Ph ^{xy}=\Rt ^{xy}$ for $\ph$-almost every $(x,y)$. 
\end{enumerate}
\end{theorem}

The proof of Theorem \ref{res-0a}-(1) is done at the end of Section \ref{sec-Sch} and the proof of Theorem \ref{res-0a}-(2-3-4) is done at the end of Section \ref{sec-convergence}.

As  a corollary, we obtain the following result.

\begin{theorem}\label{res-11}
For any $x,y\in \XX$ such that $E _{R ^{xy}}(\ell)<\infty,$  the sequence $(R ^{k,xy})_{k\ge1}$ of bridges of $(R^k)_{k\ge 1}$ is convergent and $\Lim k R ^{k,xy}=\Rt ^{xy}.$
\end{theorem}

\begin{proof}
Under the marginal constraints $\mu_0= \delta_x$ and $\mu_1=\delta_y,$ we have for all $k\ge2,$ $\ph^k=\ph=\delta _{(x,y)}$ and $\Ph^k=R ^{k,xy}$ by \eqref{eq-23}. It remains to apply Theorem \ref{res-0a}.
\end{proof}

We need some additional preliminary material to describe the dynamics of $\Ph$ and of the bridge $\Rt ^{xy}.$
Recall that a directed tree is a directed graph $(\ZZ,\to)$ that contains no circuit (directed loop). 
We denote $z\to z'$ when  the directed edge $(z,z')\in\ZZ^2$ exists and we define the order relation $\preceq$ by: $z\preceq z'$ if $z=z'$ or if there exists a finite path $z=z_1\to z_2\to\cdots\to z_n=z'.$
\\
Unlike  the following configuration (a),
\vskip 0,5cm
\begin{center}
\scalebox{1} 
{
\begin{pspicture}(0,-1.388496)(10.0,1.3684961)
\psdots[dotsize=0.16](0.08,0.4884961)
\psdots[dotsize=0.16](1.8,1.268496)
\psdots[dotsize=0.16](3.66,0.4884961)
\psline[linewidth=0.04cm,arrowsize=0.113cm 2.0,arrowlength=1.4,arrowinset=0.4]{->}(1.86,1.2284961)(3.62,0.4884961)
\psline[linewidth=0.04cm,arrowsize=0.113cm 2.0,arrowlength=1.4,arrowinset=0.4]{->}(0.14,0.54849607)(1.78,1.3084961)
\psline[linewidth=0.04cm,arrowsize=0.113cm 2.0,arrowlength=1.4,arrowinset=0.4]{<-}(0.04,0.46849608)(1.58,-0.5715039)
\psline[linewidth=0.04cm,arrowsize=0.113cm 2.0,arrowlength=1.4,arrowinset=0.4]{<-}(2.96,-0.2315039)(3.66,0.4484961)
\psdots[dotsize=0.16](2.94,-0.2315039)
\psdots[dotsize=0.16](1.62,-0.5715039)
\psline[linewidth=0.04cm,arrowsize=0.113cm 2.0,arrowlength=1.4,arrowinset=0.4]{<-}(1.68,-0.5715039)(2.86,-0.2515039)
\psdots[dotsize=0.16](6.32,0.46849608)
\psdots[dotsize=0.16](8.04,1.248496)
\psdots[dotsize=0.16](9.9,0.46849608)
\psline[linewidth=0.04cm,arrowsize=0.113cm 2.0,arrowlength=1.4,arrowinset=0.4]{->}(8.1,1.2084961)(9.86,0.46849608)
\psline[linewidth=0.04cm,arrowsize=0.113cm 2.0,arrowlength=1.4,arrowinset=0.4]{->}(6.38,0.5284961)(8.02,1.2884961)
\psline[linewidth=0.04cm,arrowsize=0.113cm 2.0,arrowlength=1.4,arrowinset=0.4]{->}(6.28,0.4484961)(7.82,-0.5915039)
\psline[linewidth=0.04cm,arrowsize=0.113cm 2.0,arrowlength=1.4,arrowinset=0.4]{->}(9.2,-0.2515039)(9.9,0.4284961)
\psdots[dotsize=0.16](9.18,-0.2515039)
\psdots[dotsize=0.16](7.86,-0.5915039)
\psline[linewidth=0.04cm,arrowsize=0.113cm 2.0,arrowlength=1.4,arrowinset=0.4]{->}(7.92,-0.5915039)(9.1,-0.2715039)
\usefont{T1}{ptm}{m}{n}
\rput(1.8451757,-1.1665039){(a)}
\usefont{T1}{ptm}{m}{n}
\rput(8.295176,-1.1665039){(b)}
\end{pspicture} 
}
\end{center}
configuration (b)  is not a circuit and it may enter a directed tree.

We  have in mind the directed tree $(\Gamma ^{xy}(\ii),\to)$ related to the set $\Gamma ^{xy}$ of all the geodesics from $x$ to $y$ on  $(\XX,\sim)$, where $z\to z'\in \Gamma ^{xy}(\ii)$ if $z\sim z'\in\XX$ and there are some $\gamma\in \Gamma ^{xy}$ and $0\le t<t'\le 1$ such that $\gamma _{t}=z$ and $\gamma _{t'}=z'.$ This tree describes the successive occurrence of the states which are visited by the geodesics from $x$ to $y$. It keeps the information of the order of occurrence, but it is regardless of the instants of jump.

\begin{theorem}[The dynamics of $\Rt ^{xy}$]\label{res-0b}
The Hypotheses \ref{hyp-00} are assumed to hold.
\begin{enumerate}
\item
Although $\Rt$ is not Markov in general, for every $x,y\in\XX,$ its bridge $\Rt ^{xy}$ is Markov.
\item
For every $x,y\in\XX,$ the jump kernel of the Markov measure $\Rt ^{xy}$ is given 
 by 
\begin{equation*}
\Jty_{t,z}=\sum _{w\in \left\{z\to\cdot\right\}^y }\frac{g^y_t(w)}{g^y_t(z)}J_{t,z}(w)\,\delta_w,\qquad 0\le t<1,\ z\in \Gamma ^{xy}(\ii),
\end{equation*}
where $\left\{z\to\cdot\right\}^y :=\left\{w\in \Gamma ^{zy}(\ii); z\to w\right\} $ is the set of all successors of $z$ in the directed tree $(\Gamma ^{zy}(\ii),\to)$ and
\begin{equation*}
g^y_t(z):= E_R \left[\exp \left(\int_t^1 J _{s,X_s}(\XX)\,ds\right) \1 _{\Gamma(t,z;1,y)} \mid X_t=z\right] 
\end{equation*}
with $$\Gamma(t,z;1,y):=\left\{\omega\in \OO; \omega_{|[t,1]}=\gamma _{|[t,1]} \textrm{ for some }\gamma\in \Gamma,\omega _{t}=z,\omega _{1}=y\right\},$$ the set of all geodesics from $z$ to $y$ on the time interval $[t,1].$
\end{enumerate}•

\end{theorem}

Remark that since $\Rt ^{xy}$ only visits $\Gamma ^{xy}(\ii),$ one can put $\Jty_{t,z}=0$ for any $z\not\in \Gamma ^{xy}(\ii)$.

The proof of Theorem \ref{res-0b} is given at Section \ref{sec-dyn}.

\begin{theorem}[The dynamics of $\Ph$]\label{res-0c}
The Hypotheses \ref{hyp-00} are assumed to hold.

\begin{enumerate}
\item
The limiting random walk $\Ph$ is Markov.
\item
The jump kernel  of the Markov measure $\Ph\in\PO$ is given by
\begin{equation*}
\Jh _{t,z}(\cdot)=\IX \Jty _{t,z}(\cdot)\, \Ph(X_1\in dy|X_t=z).
\end{equation*}
It is a mixture of the jump kernels $\Jty $ of $\Rt ^{xy},$ see Theorem \ref{res-0b}.
\end{enumerate}
\end{theorem}

The statement of Theorem \ref{res-0c}-(1) is the content of Proposition \ref{res-21} which is proved at Section \ref{sec-dyn}. The second statement is proved at the end of Section \ref{sec-dyn}.

Gathering Theorems \ref{res-0b} and \ref{res-0c}, one obtains for all $t$ and $z$ such that $\mu_t(z)>0,$ 
\begin{equation}\label{eq-63b}
\begin{split}	
\Jh _{t,z}(\cdot)=\IX \Bigg(\sum _{w\in \{z\to\cdot\}^y }\frac{E_R [\exp (\int_t^1 J _{s,X_s}(\XX)\,ds) \1 _{\Gamma(t,w;1,y)} \mid X_t=w]}{E_R \Big[\exp (\int_t^1 J _{s,X_s}(\XX)\,ds) \1 _{\Gamma(t,z;1,y)} \mid X_t=z\Big]}J_{t,z}(w)\,\delta_w(\cdot)\Bigg) \\
\times\quad \Ph(X_1\in dy|X_t=z).
\end{split}
\end{equation}

\section{Main results about  displacement interpolations}\label{sec-di}

We have defined the displacement interpolation $[\mu_0,\mu_1]$ as the marginal flow $\mu_t:=\Ph_t,$ $0\le t\le1,$ of the displacement random walk $\Ph.$
A direct consequence of Theorem \ref{res-0c} is the 
\begin{corollary}[The dynamics of $\mu$]\label{res-0d}
The Hypotheses \ref{hyp-00} are assumed to hold.
\\
The displacement interpolation $[\mu_0,\mu_1]$ solves the following evolution equation
\begin{equation*}
\left\{\begin{array}{ll}
\partial_t \mu_t(z)=\sum _{w}[\mu_t(w)\Jh _{t,w}(z)-\mu_t(z)\Jh _{t,z}(w)], & 0\le t\le 1, z\in\XX,\\
\mu_0,& t=0.
\end{array}
\right.
\end{equation*}
\end{corollary}

\subsection*{A Benamou-Brenier type formula}

Let us have a closer look at the Benamou-Brenier type formula \eqref{eq-59b}.

\begin{definition}[Change of time]
A {change of time} $\tau$ is an \emph{absolutely continuous} function $\tau:\ii\to\ii$ such that  $\tau(0)=0$, $\tau(1)=1$  and with a nonnegative generalized derivative $0\le\dot\tau\in L^1(\ii).$
\end{definition}

For any change of time $\tau$ and any measure $Q\in\MO,$ we denote $$Q^\tau:=(X_\tau)\pf Q$$
where $X_\tau(t):=X _{\tau(t)},  t\in\ii.$
For any flow $\nu$ of probability measures and any jump kernel $\mathcal{J}$, we  denote $\nu \mathcal{J}_t(dxdy):=\nu_t(dx) \mathcal{J}_{t,x}(dy)$ and  $\nu \mathcal{J}_t(\XXX):=\IXX \nu_t(dx)\mathcal{J} _{t,x}(dy)$. 

The following theorem is a consequence of the Markov property of $\Ph$ which was stated at Theorem \ref{res-0c}.

\begin{theorem}[A Benamou-Brenier type formula]\label{res-0e}
Suppose that the Hypotheses \ref{hyp-00} are satisfied.
\begin{enumerate}
\item
We have
\begin{equation}\label{eq-59}
W_1(\mu_0,\mu_1)
= \inf _{\nu,\JJ}\Iii dt \IXX d(z,w)\,\nu\JJ _{t}(dzdw)<\infty 
\end{equation}
where the infimum is taken over all  couples $(\nu,\JJ)$ such that $\nu=(\nu_t)_{t\in\ii}\in \PX ^{\ii}$ is a time-differentiable  flow of probability measures on $\XX,$ $\JJ=(\JJ _{t,z})_{t\in\ii,z\in\XX}\in \MX ^{\ii\times\XX}$ is a measurable jump kernel, $\nu$ and $\JJ$ are linked by
\begin{equation}\label{eq-45}
\left\{
\begin{array}{l}
\partial_t \nu_t(z)+\IX [\nu_t(z)\JJ _{t,z}(dw)-\nu_t(dw)\JJ _{t,w}(z)]=0,\qquad 0<t<1,\ z\in\XX\\
\nu_0=\mu_0,\ \nu_1=\mu_1
\end{array}
\right.
\end{equation}
with $\IX \nu_t(z)\JJ _{t,z}(dw)<\infty$ for all $0<t<1,\ z\in\XX.$
\item
The infimum
 $\inf _{\nu,\JJ}$ in \eqref{eq-59} is attained at $(\mu,\Jh)$ where $\mu$ is  the displacement interpolation and $\Jh$ is the jump kernel of $\Ph$. Hence
\begin{equation*}
W_1(\mu_0,\mu_1)= \Iii dt \IXX d(z,w)\,\mu\Jh _{t}(dzdw).
\end{equation*}

\item
For any   change of time $\tau$,   the infimum
 $\inf _{\nu,\JJ}$ in \eqref{eq-59} is also attained at $(\mu^\tau,\Jh^\tau)$ where $\mu^\tau$ is  the displacement interpolation and $\Jh^\tau$ is the jump kernel that are  associated  with $\Ph^\tau.$ Moreover, $\Ph^\tau$  is the displacement random walk associated with $R^\tau,$ i.e.\  the analogue of $\Ph$ when $R$ is replaced  with $R^\tau$ 
 and
\begin{equation*}
W_1(\mu_0,\mu_1)= \Iii dt\IXX d(z,w)\,\mu\Jh^\tau _{t}(dzdw)
\end{equation*}
where for  almost every $t\in\ii$, $\mu\Jh_t^\tau:=\dot\tau(t)\mu\Jh _{\tau(t)}$  is the mass displacement distribution of $\Ph^\tau$ at time $t$.
\end{enumerate}
\end{theorem}

\begin{proof}
\boulette{(1) and (2)}
With \eqref{eq-97} and Theorem \ref{res-0a}-(3),  we have
\begin{equation*}
W_1(\mu_0,\mu_1)=\inf \eqref{MK}=\inf \eqref{MKdyn}=E _{\Ph}(\ell).
\end{equation*}
But,  for any Markov random walk $P\in\PO$ on $(\XX,\sim)$ with jump kernel $\JJ$ and such that $E_P(\ell)<\infty,$ we have
\begin{eqnarray*}
E_P(\ell)=E_P\IiX d(X_t,y)\, \JJ _{t,X_t}(dy)dt
=\Iii dt\IXX d(z,w)\,P_t(dz)\JJ _{t,z}(dw).
\end{eqnarray*}
This proves that $W_1(\mu_0,\mu_1)=\inf \left\{E_P(\ell);P \textrm{ Markov on }(\XX,\sim): P_0=\mu_0,P_1=\mu_1\right\}=E _{\Ph} (\ell),$ which is the announced result since $\nu_t:=P_t$ solves the Fokker-Planck equation in \eqref{eq-45}.

\Boulette{(3)}   
We denote $P^*$  the displacement random walk  associated to \eqref{Skdyn}$_{k\ge 2}$ with $R^\tau$ instead of $R$.
As $X_\tau$ is injective, we have $H(P|R^k)=H(P^\tau|R ^{\tau,k})$ for all $P\in\PO$ and $k\ge1$. This implies that   $P^*=\Ph^\tau$. Hence  (3) follows from (2).
\end{proof}

\subsection*{Constant speed and minimizing  displacement interpolations}

Next proposition is another consequence of the Markov property of $\Ph$.

\begin{proposition}\label{res-42}
For all $0\le s\le t\le1,$ let $\Ph _{st}:=(X_s,X_t)\pf \Ph\in\PXX$ be the joint law of the positions at time $s$ and $t$ under $\Ph.$ Then,
\begin{enumerate}
\item
$\Ph _{st}\in\PXX$ is an optimal coupling of $\mu_s$ and $\mu_t,$ meaning that $\Ph _{st}$ is a solution of \eqref{MK} with $\mu_s$ and $\mu_t$ as  prescribed marginal constraints;
\item
$W_1(\mu_s,\mu_t)=\int _{[s,t]}dr\IXX d(z,w)\,\mu\Jh_r(dzdw).$
\end{enumerate}
\end{proposition}

\begin{proof}
Both statements are  consequences of 
\begin{itemize}
\item
the Markov property of $\Ph,$ see Theorem \ref{res-0c}, which allows for surgery by gluing the bridges of $\Ph _{[s,t]}$ together with the restrictions $\Ph _{[0,s]}$ and $\Ph _{[t,1]}$, where we denote $P _{[u,v]}:=(X_t;u\le t\le v)\pf P$;
\item
the fact that $\ell _{st}:=\sum _{s<r<t}d(X _{r^-},X_r)$ is insensitive to  changes of time: i.e.\ for any strictly increasing mapping $\theta:[s,t]\to\ii$ with $\theta(s)=0,$ $\theta(t)=1,$ we have $\ell _{st}=\ell _{01}(X_\theta).$
\end{itemize}
A standard ad absurdum reasoning leads to (1).
Statement (2) follows from (1), a change of variables formula based on any absolutely continuous change of time $\theta:[s,t]\to\ii$ and the general identity
\begin{equation}\label{eq-41}
\mathcal{J}^{\theta}_r=\dot\theta(r) \mathcal{J}_{\theta(r)},\quad \textrm{for almost every }\ r\in (s,t)
\end{equation}
where $(\mathcal{J}_u; u\in\ii)$ is any jump kernel and $(\mathcal{J}^{\theta}_r; r\in[s,t])$ the jump kernel resulting from the mapping $X _{\theta}.$
\end{proof}

Proposition \ref{res-42}-(2) entitles us to define the speed of the displacement interpolation $\mu$ at time $t$ by
\begin{equation}\label{eq-61}
\mathrm{speed}(\mu)_t:=\IXX d(z,w)\, \mu\Jh_t(dzdw),\quad 0\le t\le1,
\end{equation}
to obtain
\begin{equation*}
W_1(\mu_0,\mu_1)=\Iii \mathrm{speed}(\mu)_t\,dt.
\end{equation*}
For any change of time $\tau:\ii\to\ii,$ we see with this identity,  \eqref{eq-41} and the change of variable formula that
\begin{equation*}
W_1(\mu ^\tau_0,\mu^\tau_1)=W_1(\mu_0,\mu_1).
\end{equation*}
Hence, there are infinitely many $\mu ^{\tau}$ that minimize the action in formula \eqref{eq-59}. This is the content of Theorem \ref{res-0e}-(3). On the other hand, next result states that only  one of them has a constant speed.

\begin{proposition}
Under the Hypotheses \ref{hyp-00},
there exists a unique change of time $\tau_o$ such that $\mu ^{\tau_o}$ has a constant speed, i.e.
\begin{equation*}
W_1(\mu ^{\tau_o}_s,\mu_t ^{\tau_o})=(t-s)W_1(\mu_0,\mu_1),\quad \forall 0\le s\le t\le1.
\end{equation*}
\end{proposition}

\begin{proof}
Indeed, this equation is equivalent to 
\begin{equation}\label{eq-64}
\dot\tau(s) \psi(\tau(s))=W_1(\mu_0,\mu_1),\quad \textrm{a.e.},
\end{equation}
where 
$	
\psi(t):=\IXX d(z,w)\,\mu\Jh _{t}(dzdw),$ a.e.	
 Clearly, the assumption that $d$ is uniformly lower bounded and \eqref{eq-63b} imply that $\psi>0.$ Hence, a solution of \eqref{eq-64} is given by
\begin{equation}\label{eq-95}
\tau_o(s)=\Psi _{\mu_0,\mu_1} ^{-1}(W_1(\mu_0,\mu_1) \,s),\quad s\in\ii
\end{equation}
where for all $t\in\ii,$
\begin{equation*}
0\le \Psi _{\mu_0,\mu_1}(t):=\int _{[0,t]} \psi(r)\,dr=\int _{[0,t]}dr\int_{\XXX} d(z,w)\,\mu\Jh _{r}(dzdw)\le W_1(\mu_0,\mu_1)<\infty.
\end{equation*}
Let us prove the uniqueness. Remark that, as a continuous strictly monotone function, $\tau_o$ is bijective. In addition, it is absolutely continuous. Hence, any change of time $\tau$ is equal to $\tau_o\circ \sigma$ for some change of time $\sigma.$ Now, instead of starting from $\mu,$ let us do a change of time $\sigma$ on $\mu ^{\tau_o}.$ Defining $\psi_o(u):=\IXX d(z,w)\,\mu\Jh ^{\tau_o} _{u}(dzdw),$ a.e.\ instead of $\psi,$ we arrive similarly at $\dot \sigma(u) \psi_o(\sigma(u))=W_1(\mu_0,\mu_1),$ a.e. But, $\psi_o(u)=W_1(\mu_0,\mu_1)$ for all $u$. Hence, $\dot \sigma=1,$ from which the desired result follows.
\end{proof}

\begin{definition}[Constant speed displacement interpolation]\label{def-06}
The time changed displacement interpolation $\mu ^{\tau_o}$ with $\tau_o$ given at \eqref{eq-95} is called a constant speed displacement interpolation.
\end{definition}
As a definition, a constant speed displacement interpolation is a minimizing geodesic between $\mu_0$ and $\mu_1$ in the Wasserstein space $(\PuX,W_1)$. But the general definition of a minimizing geodesic  is misleading in the present non-strictly convex setting, since there are infinitely many \emph{action minimizing} $W_1$-geodesics which do not have a constant speed.

One must be aware that, in general, the change of time $\tau_o$ depends on $\mu_0$ and $\mu_1$. Nevertheless, we shall see below at Theorem \ref{res-0f} that in the special important case where the distance $d$ is the standard graph distance $d_\sim$ specified by
\begin{equation}\label{eq-96}
\forall x,y\in\XX,\quad d_\sim(x,y)=1 \iff x\sim y,
\end{equation}
for any $\mu_0,\mu_1,$ the displacement interpolation $[\mu_0,\mu_1]$ has a constant speed.

Let us go back to a general discrete metric measure  graph  $(\XX,\sim,d,m)$ where $d$ might differ from the standard graph distance $d_\sim.$
When transferring analogically the Lott-Sturm-Villani theory to this graph setting, it is likely that one should consider a reference random walk $R$ that satisfies the detailed balance conditions \eqref{eq-94}:
\begin{equation*}
m_xJ _{t,x}(y)=m_y J _{t,y}(x),\quad \forall x,y\in\XX, t\in\ii,
\end{equation*}
 for being respectful of the metric measure  graph structure. Indeed, this implies that  for all $0\le t\le1$ and $k\ge1,$ $R_t^k=m$ and also that the generators $L_t^k$, see \eqref{eq-85}, are self-adjoint on $L^2(\XX,m).$
\\ 
Pick $\mu_0$ and $\mu_1$ and consider the associated constant speed interpolation $\mu ^{\tau_o}.$ Since, for any change of time $\tau,$ the mapping $Q\in\MO\mapsto Q ^{\tau}\in\MO$ is injective, we have $H(P ^{\tau}|R ^{\tau})=H(P|R)$ for any $P\in\PO.$ This implies that $\mu ^{\tau_o}$ is the displacement interpolation associated with $R ^{\tau_o}:$
\begin{equation*}
\mu ^{\tau_o}=[\mu_0,\mu_1]^{R}\circ \tau_o=[\mu_0,\mu_1]^{R ^{\tau_o}}.
\end{equation*}
But, since $J^\tau_t=\dot \tau(t)J _{\tau(t)},$ $R ^{\tau_o}$ inherits of the detailed balance condition satisfied by $R:$
\begin{equation*}
m_xJ ^{\tau_o} _{t,x}(y)=m_y J ^{\tau_o}_{t,y}(x),\quad \forall x,y\in\XX, t\in\ii.
\end{equation*}
Therefore, the constant speed displacement interpolation $\mu ^{\tau_o}$ is really a displacement interpolation in the sense of Definition \ref{def-03} with respect to the reference random walk $R ^{\tau_o}$ which is respectful of the discrete metric measure graph $(\XX,\sim,d,m).$

\subsection*{Conservation of the average rate of mass displacement}

Next result tells us that along any displacement interpolation $[\mu_0,\mu_1],$ the average rate of mass displacement, as defined below, doesn't depend on time.

\begin{definitions}[Rate of mass displacement]\label{def-05}
For any Markov random walk $P\in\PO$ with jump kernel $(\JJ _{t,x};t\in\ii, x\in\XX)$, we denote $\nu_t=P_t\in\PX$ and call
\begin{equation*}
\nu\JJ_t(dxdy):= \nu_t(dx)\JJ _{t,x}(dy),\qquad 0\le t\le1
\end{equation*} 
the \emph{distribution of the rate of  mass displacement}  of $P$ at time $t$.
\\
We also call
\begin{equation*}
\nu\JJ_t(\XXX):=\IXX \nu_t(dx)\JJ _{t,x}(dy),\qquad 0\le t\le1
\end{equation*} 
the \emph{average rate of mass displacement} of $P$ at time $t$.
\end{definitions}

\begin{theorem}[Conservation of the average rate of mass displacement]\label{res-0f}
Suppose that the Hypotheses \ref{hyp-00} are satisfied. Let   $\Jh$ be the  jump kernel of the displacement random walk $\Ph$ and $\mu$ the corresponding displacement interpolation. There exists some $K>0$ such that
\begin{equation*}
\mu\Jh_t(\XXX)=K,\quad \forall t\in\ii.
\end{equation*}
In particular, when the distance $d$ is the standard discrete distance $d _{\sim}$, see \eqref{eq-96}, the displacement interpolation $\mu$ has a constant speed.
\end{theorem}

Last statement simply relies on the remark that when $d=d_\sim,$ the speed of $\mu$ coincides with its average rate of mass displacement.

Theorem \ref{res-0f} is a restatement of Theorem \ref{res-0fb} which is proved at Section \ref{sec-conservation}.

\begin{corollary}
The constant speed displacement interpolation $\mu ^{\tau^o}$ defined at Definition \ref{def-06} has also a constant average rate of mass displacement.
\end{corollary}

\begin{proof}
Since  $\mu ^{\tau^o}=[\mu_0,\mu_1] ^{R ^{\tau^o}}$ is the $R ^{\tau^o}$-displacement interpolation, one can apply Theorem \ref{res-0f}.
\end{proof}

\subsection*{Natural substitutes for the constant speed geodesics on a discrete metric graph}

Let $R$ be given. When specifying $\mu_0=\delta_x$ and $\mu_1=\delta_y,$ the displacement random walk $\Ph$ is simply $\Rt ^{xy}.$ Moreover, there exists a unique change of time $\tau ^{xy}$ such that
\begin{equation*}
\mu ^{xy}:=[\delta_x,\delta_y]\circ \tau ^{xy}=\Rt ^{xy}_{\tau ^{xy}}
\end{equation*}
has a constant speed. Its dynamics is given by the current equation
\begin{equation*}
\left\{\begin{array}{ll}
\partial_t \mu ^{xy}_t(z)-{\dot\tau ^{xy}_t}\sum _{w}[\mu ^{xy}_t(w)\Jty _{\tau ^{xy}_t,w}(z)-\mu ^{xy}_t(z)\Jty _{\tau ^{xy}_t,z}(w)]=0, & 0\le t\le 1, z\in\XX,\\
\mu ^{xy}_0=\delta_x,& t=0.
\end{array}
\right.
\end{equation*}
The $(R,d)$-displacement interpolation $\mu ^{xy}$ is a natural time-continuous averaging of the piecewise constant paths $t\mapsto \delta _{\gamma(t)}$ with $\gamma$ in the set  $\Gxy$ of all  $d$-geodesics joining $x$ and $y$. It depends on the choice of the reference random walk $R.$

\subsection*{Binomial interpolations}

We present some easy examples of constant speed  interpolations $\mu ^{xy}.$
Let us consider the simplest and important setting where $R$ is the reversible simple random walk and  the distance $d=d_\sim$ is the standard graph distance.
 The jump kernel is described at \eqref{eq-98a}: $J_x(y)=1/n_x,$ $\forall x\sim y,$ and with the initial ``volume measure'' given at \eqref{eq-98b}: $m_x=n_x,$ $\forall x\in\XX$. 
\\
Let $x$ and $y$ be fixed . We know by Theorem  \ref{res-0f} that the displacement interpolation $\mu ^{xy}=[\delta_x,\delta_y]=(G ^{xy}_t) _{0\le t\le1}$ has a constant speed. The dynamics of $G ^{xy}$ is specified at Theorem \ref{res-0b}, for any $t\in[0,1),$ $z\in \Gamma ^{xy}(\ii)$ and $w\in \left\{z\to\cdot\right\}^y,$ by 
\begin{equation*}
J ^{G,y}_{t,z}(w)=\1 _{\left\{z\ne y\right\} }n_z ^{-1}\frac{g_t^y(w)}{g_t^y(z)}=\1 _{\left\{z\ne y\right\} }n_z ^{-1}\frac{R(\Gamma(t,w;1,y)\mid X_t=w)}{R(\Gamma(t,z;1,y)\mid X_t=z)}
\end{equation*}

\subsubsection*{The complete graph} Let $\XX=\left\{1,\dots,n\right\} $ with $x\sim y$ for all $x\ne y\in\XX.$ Then, for all $x\ne y$ and all $0\le t<1,$ we have $J ^{G,y}_{t,x}(y)=1/(1-t)$ and the probability that no jump occurred before time $t$ is $\mathrm{Proba}(N _{\lambda(t)}=0)$ where $N_\lambda$ denotes a random variable distributed according to Poisson($\lambda$) and $\lambda(t)=\int_0 ^{t}\frac{1}{1-s}\,ds=-\log(1-t). $ Therefore, $\mathrm{Proba}(N _{\lambda(t)}=0)=\exp(-\lambda(t))=1-t$ and $$G ^{xy}_t=(1-t)\delta_x+t \delta_y=\sum _{z\in \left\{x,y\right\} } t ^{d(x,z)} (1-t)^{d(z,y)}\, \delta_z,$$ for any $0\le t\le1.$ This law is in one-one correspondence with the Bernoulli  law $\mathcal{B}(t)$ which  is the specific binomial law $\mathcal{B}(2,t).$

\subsubsection*{The graph $\mathbb{Z}$}

We  consider the simple situation where $(\XX,\sim)$ is the set of integers $\mathbb{Z}$ with its natural graph structure. The reference random walk $R$ is the simple walk with $J_z=(\delta _{z-1}+\delta _{z+1})/2,$ $z\in \mathbb{Z},$ and the counting measure as its initial measure. Take $x<y\in \mathbb{Z}.$ Then, for any $0\le t<1$ and $x\le z <y,$ denoting $N _{1-t}$ a random variable distributed according to the Poisson($1-t$) law and $\delta=d(z,y)=y-z$, we obtain
\begin{eqnarray*}
J ^{G,y}_{t,z}(z+1)&=& \frac{1}{2}\frac{\mathrm{Proba}(N _{1-t}=d(z+1,y))(1/2)^{d(z+1,y)}}{\mathrm{Proba}(N _{1-t}=d(z,y))(1/2)^{d(z,y)}}\\
	&=&\frac{1}{2}\frac{e ^{1-t}2^{-(\delta-1)}(1-t)^{\delta-1}/(\delta-1)!}{e ^{1-t}2^{-\delta}(1-t) ^{\delta}/\delta!}
=d(z,y)/(1-t).
\end{eqnarray*}

\begin{claim}\label{res-45}
This proves that $G ^{xy}$ has the same law as $x+\widetilde N$ where $\widetilde N$ is the bridge of a Poisson process $(N_t) _{0\le t\le 1},$ that is $\mathrm{Proba}(G ^{xy}\in\cdot)=\mathrm{Proba}(N\in \left\{\cdot\right\} -x\mid N_1=d(x,y)).$ 
\end{claim}

Since for each $t\in\ii,$ the law of $\widetilde N_t$ is the binomial $\mathcal{B}(d(x,y),t),$ $[\delta_x,\delta_y]$ is sometimes called a binomial interpolation, see \cite{GRST12,Hil12,Hil}. For each $0<t<1,$ the support of $G ^{xy}_t$  is $\left\{x, x+1,\dots,y\right\} $ and we have
\begin{equation*}
G ^{xy}_t=\sum _{z: x\le z\le y} \begin{pmatrix}
d(x,y)\\d(x,z)
\end{pmatrix}
t ^{d(x,z)} (1-t)^{d(z,y)}\, \delta_z.
\end{equation*}

\subsubsection*{The hypercube}

Consider the hypercube $\XX=\left\{0,1\right\}^{n}$ with its natural graph structure so that the graph distance is the Hamming distance $d(x,y)=\sum _{1\le i\le n}\1 _{\left\{x_i\ne y_i\right\} }$ where $x=(x_1,\dots,x_n)$ and $y=(y_1,\dots,y_n).$ The reference path measure $R$ is the simple random walk with the uniform measure as its initial law.
The directed tree that describes the geodesic dynamics between $x$ and $y$ has exactly $d(x,y)!$ directed chains with length $d(x,y)$ and endpoints $x$ and $y$. The law of $G ^{xy}$ is the uniform mixture of the $d(x,y)!$ corresponding Poisson bridges. 
\\
In order to describe for any $0\le t\le1,$ the law $G ^{xy}_t,$ let us encode  each intermediate state by an ordered sequence in $\left\{\textsf{d},\textsf{s}\right\}^{d(x,y)}$ where \textsf{d} and \textsf{s} stand respectively for ``different'' and ``same''. With this encoding,  $(\textsf{d},\dots,\textsf{d})$ is $x$ since $x$ has $d(x,y)$ components that are different from $y$, of course  $(\textsf{s},\dots,\textsf{s})$ is $y$  and
we see that the support $S ^{xy}$ of $G ^{xy}_t$ consists of $2 ^{d(x,y)}$ intermediate states  at any time $0<t<1.$ A short computation shows that for each $0<t<1,$ we have
\begin{equation*}
G ^{xy}_t=\sum _{z\in S ^{xy}} t ^{d(x,z)} (1-t)^{d(z,y)}\, \delta_z.
\end{equation*}

\section{The Schr\"odinger problem}\label{sec-Sch}

The reversing measure $m$ of  the simple random walk on an infinite graph is  unbounded, see \eqref{eq-98b}. Since it is the analogue of the volume measure of a Riemannian manifold, it is likely that the relative entropy with respect to $m$ should play an important role when trying to develop a Lott-Sturm-Villani theory on infinite graphs. Consequently, in this case the reference path measure $R$ is unbounded.
\\
In order to state the Schr\"odinger problem, it will be necessary to have in mind  some basic facts about relative entropy with respect to a possibly \emph{unbounded} reference measure. They are collected at Section \ref{sec-ent}.

\subsection*{Schr\"odinger problem}

We briefly introduce the main features of the Schr\"odinger problem. For more detail, see for instance the survey paper \cite{Leo12e}.
\\
The dynamic Schr\"odinger problem associated with the  random walk $R\in\MO$ is the following entropic minimization problem
\begin{equation*}\label{Sdyn}
H(P| R)\to \textrm{min};\qquad P\in\PO: P_0=\mu_0,\ P_1=\mu_1
 \tag{S$_{\mathrm{dyn}}$}
\end{equation*}
where $\mu_0, \mu_1\in\PX$ are prescribed initial and final marginals. As a \emph{strictly} convex problem, it admits at most one solution.

Let us  particularize the consequences of the additivity formula \eqref{eq-33} to $r=R,$ $p=P$ and $\phi=(X_0,X_1).$ Denoting $Q _{01}=(X_0,X_1)\pf Q\in \mathrm{M}_+(\XXX)$ and $Q ^{xy}=Q(\cdot\mid X_0=x,X_1=y)\in\PO,$   we have for all $P\in\PO,$
\begin{equation}\label{eq-21}
H(P|R)=H(P _{01}|R _{01})+\IXX H(P ^{xy}|R ^{xy})\,P _{01}(dxdy)
\end{equation}
which implies that
$
H(P _{01}|R _{01})\le H(P|R)
$
 with equality (when $H(P|R)<\infty$) if and only if 
\begin{equation}\label{eq-23}
P ^{xy}=R ^{xy}
\end{equation}
for $P _{01}$-almost every $(x,y)\in\XXX,$ see \eqref{eq-34}. Therefore $\Ph$ is the (unique) solution of \eqref{Sdyn} if and only if it disintegrates as 
\begin{equation}\label{eq-11}
\Ph(\cdot)=\IXX R ^{xy}(\cdot)\,\ph(dxdy)\in\PO
\end{equation}
where $\ph\in\PXX$ is the (unique) solution of the minimization problem
\begin{equation*}\label{S}
H(\pi| R _{01})\to \textrm{min};\qquad \pi\in\PXX: \pi_0=\mu_0,\ \pi_1=\mu_1
 \tag{S}
\end{equation*}
where $\pi_0,\pi_1\in\PX$ are respectively the first and second marginals of $\pi\in\PXX.$ Identity \eqref{eq-11} means that 
\begin{itemize}
\item
$\Ph$ shares its bridges with the reference path measure $R$, i.e.\ \eqref{eq-23};
\item
these bridges are mixed according to  $$\ph=\Ph _{01},$$ the unique solution of \eqref{S}.
\end{itemize}
The entropic minimization problem \eqref{S} is called the (static) \emph{Schr\"odinger problem}.
\\
With \eqref{eq-11}, we see that
\begin{equation}\label{eq-13}
\inf \eqref{Sdyn}=\inf \eqref{S}\in (-\infty,\infty].
\end{equation}

\subsection*{Proofs of Theorem  \ref{res-0a}-(1) and Proposition \ref{res-43}}

We begin with a key technical statement.

\subsubsection*{Girsanov's formula}

We shall take advantage, several times in the remainder of the article, of the absolute continuity of $R^k$ with respect to $R$.
Girsanov's formula gives the expression of the Radon-Nykodim derivative of $R^k$ with respect to $R$:
\begin{equation}\label{eq-14b}
Z^k:=\frac{dR^k}{dR}=
\exp \left(-(\log k) \ell+U^k\right) 
\end{equation}
where $U^k:=\IiX (1-k ^{-d(X _{t},y)})\, J _{t,X _{t}}(dy)dt$ and $\ell$ is the length defined at \eqref{eq-90}.


\subsubsection*{Proof of Theorem \ref{res-0a}-(1)}
The uniqueness follows from the strict convexity of the Schr\"odinger problem and we have just seen that $\ph=\Ph _{01}$. The Markov property of $\Ph$ which is inherited from the Markov property of $R$ is proved at \cite[Prop.\,2.10]{Leo12e}.
\\
It remains to show the existence.
For any $P\in\PO$ and any $k\ge1,$  with   \eqref{eq-14b} we see that 
\begin{eqnarray*}
H(P|R^k)&=&H(P|R)+\log k\ E_P(\ell) -E_P U^k\nonumber\\&\le& H(P|R)+\log k\ E_P(\ell)\nonumber\\
	&=& H(P _{01}|R _{01})+\IXX H(P ^{xy}|R ^{xy})\,P _{01}(dxdy)+\log k\ E_P(\ell).
\end{eqnarray*}
Choosing
\begin{equation}\label{eq-00}
P^o(\cdot):=\IXX R ^{xy}(\cdot)\,\pi^o(dxdy),
\end{equation}
we have
\begin{equation}\label{eq-65}
H(P^o|R^k)\le H(\pi^o|R _{01})+\log k\ \IXX E _{R ^{xy}}(\ell)\,\pi^o(dxdy).
\end{equation}
With Hypothesis \ref{hyp-00}-($\mu$) we obtain $\inf \eqref{Skdyn}<\infty$ and it follows that \eqref{Skdyn} and \eqref{Sk} admit a solution, see \cite[Lemma\,2.4]{Leo12e}.
\hfill$\square$

\subsubsection*{Proof of Proposition \ref{res-43}}
Taking $\pi^o=\mu_0\otimes \mu_1$ in \eqref{eq-65} gives
\begin{equation*}
H(P^o|R^k)\le H(\mu_0\!\otimes\!\mu_1|R _{01})+\log k\ \IXX E _{R ^{xy}}(\ell)\,\mu_0(dx)\mu_1(dy).
\end{equation*}
It is proved at \cite[Prop.\,2.5]{Leo12e} that  the assumptions (i), (ii) and (iv) together with $H(\mu_0|R_0), H(\mu_1|R_1)<\infty$ imply  $H(\mu_0\otimes  \mu_1|R _{01})<\infty$. As regards the last term, it is clear that (i) and (iii) imply $\IXX E _{R ^{xy}}(\ell)\,\mu_0(dx)\mu_1(dy)<\infty.$ \hfill$\square$

\section{Lazy random walks converge to displacement random walks}\label{sec-convergence}

The aim of this section is to make precise the convergence of  \eqref{Skdyn}$_{k\ge2}$. It is proved at Theorem \ref{res-0a} that the sequence  of minimizers of \eqref{Skdyn}$_{k\ge2}$  has a limit in $\PO$ which is  singled out among the infinitely solutions of the dynamic Monge-Kantorovich problem \eqref{MKdyn}.
 As a corollary, we describe at Theorem \ref{res-11} the convergence of  the sequence   of bridges $(R ^{k,xy})_{k\ge1}$.

\subsection*{The topological  path space $\OO$}
The countable set $\XX$ is equipped with its discrete topology. The  set $D(\ii,\XX)$ of all left-limited  right-continuous paths on $[0,1)$ and left-continuous at the terminal time $t=1,$ be equipped with the Skorokhod topology.  Note that, although for any $0<t<1$, the mapping $X_t:D(\ii,\XX)\to\XX$ is discontinuous, both the endpoint positions $X_0$ and $X_1$ are continuous. This will be used later several times.
Let us denote the total number of jumps, defined on $D(\ii,\XX)$, by
\begin{equation}\label{eq-17}
N:=\st \1 _{\left\{X _{t^-}\not=X_t\right\} }\in\mathbb{N}\cup \left\{\infty\right\}.
\end{equation}
We consider $\OOt=\left\{\omega\in D([0,1],\XX); \forall t\in(0,1), \omega _{t^-}\not=\omega_t\implies \omega _{t^-}\sim \omega_t\right\} $  the subset of all paths compatible with the graph structure and introduce 
\begin{equation}\label{eq-19}
\OO:=\left\{N<\infty\right\}\cap \OOt,
\end{equation}
the set of all c\`adl\`ag paths  from $\ii$ to $\XX$   which are compatible with the graph structure and piecewise constant with finitely many jumps. 

Our Hypotheses \ref{hyp-00}, and in particular \eqref{eq-01b}, imply that the support of each $R^k$ is included in $\OO.$ Since, $\Ph^k$ is absolutely continous with respect to $R^k,$ it also lives on $\OO$ which appears to be the relevant path space.

As $\XX$ is a discrete space, for each $n\in \mathbb{N},$ $\left\{N=n\right\}\cap\OOt $ is a  closed and open (clopen) set. In particular $\OO$ is closed in $D(\ii,\XX)$ and it inherits its (trace) Polish topological structure and the corresponding (trace) Borel $\sigma$-field which is generated by the canonical process (restricted to $\OO$). The path space
$ 
\OO=\bigsqcup _{n\in \mathbb{N}}\left\{N=n\right\}\cap\OOt
$ 
is partitioned by the disjoint clopen sets $\left\{N=n\right\}\cap\OOt .$
A small neighbourhood of   $\omega\in\OO$ consists of paths visiting exactly the same states as $\omega$ in the same order of occurrence and with jump times close to $\omega$'s ones. From now on, any topological statement on $\OO$ refers to this  topology and the canonical process $(X_t)_{0\le t\le1}$ lives on $\OO.$

\subsection*{$\Gamma$-convergence}
The right notion of convergence for the sequences of minimization problems   $\eqref{Skdyn}_{k\ge2}$ and $\eqref{Sk}_{k\ge2}$ is the $\Gamma$-convergence which is briefly over-viewed  now.
 Recall that $\Glim k f^k=f$ on the metric space $Y$ if and only if for any $y\in Y,$
\begin{enumerate}[(a)]
\item
$\Liminf k f^k(y_k)\ge f(y)$ for any convergent sequence $y_k\to y,$
\item
$\Lim k f^k(y^o_k)=f(y)$ for some sequence $y^o_k\to y.$
\end{enumerate}
A  function $f$ is said to be coercive if for any $a\ge \inf f,$ $\left\{f\le a\right\} $ is a compact set.
\\
 The sequence $(f^k)_{k\ge1}$ is said to be equi-coercive if for any real $a$, there exists some compact set $K_a$ such that $\cup_k\left\{f^k\le a\right\} \subset K_a.$
\\
If in addition to $\Glim k f^k=f$, the sequence $(f^k)_{k\ge1}$ is equi-coercive,
 then:
\begin{itemize}
\item
$\Lim k \inf f^k=\inf f,$
\item
if $\inf f<\infty,$
any limit point $y^*$ of a sequence $(y^*_k)_{k\ge1}$ of approximate minimizers i.e.: $f^k(y^*_k)\le \inf f^k+\epsilon_k$ with $\epsilon_k\ge0$ and $\Lim k \epsilon_k=0,$ minimizes $f$ i.e.: $f(y^*)=\inf f.$
\end{itemize}
For more detail about $\Gamma$-convergence, see \cite{DalMaso} for instance.

\subsection*{The convergences of $\eqref{Skdyn}_{k\ge2}$ and $\eqref{Sk}_{k\ge2}$}

As \eqref{Skdyn} and \eqref{Sk} are deeply linked to each other via the relations \eqref{eq-11} and \eqref{eq-13}, the convergence of the static problems will follow from the convergence of the dynamic problems \eqref{Skdyn}.

The convex indicator $\iota_A$ of  any subset $A,$ is defined to be equal to $0$ on $A$ and to $\infty$ outside $A.$
We denote for each $k\ge2,$ (we drop $k=1$ not to divide by $\log (1)$ below),
\begin{equation*}
I^k(P):=H(P|R^k)/\log k+\iota_{\{P:P_0=\mu_0,P_1=\mu_1\}},\qquad P\in\PO,
\end{equation*}
so that \eqref{Skdyn} is simply: ($I^k\to \textrm{min}).$ We also define
$$
I(P)=E_P(\ell)+\iota_{\{P:P_0=\mu_0,P_1=\mu_1\}},\qquad P\in\PO.
$$
Let us rewrite $(I\to \textrm{min})$ as the following dynamic Monge-Kantorovich problem:
\begin{equation*}
E_P(\ell)\to \textrm{min};\qquad P\in\PO: P_0=\mu_0, P_1=\mu_1.
\tag{MK$_{\mathrm{dyn}}$}
\end{equation*}
Otherwise stated, the topologies on $\PO$ and $\PXX$ are respectively the topologies of narrow convergence: $\sigma(\PO,C_b(\OO))$ and $\sigma(\PXX,C_b(\XXX))$ which are weakened by the spaces $C_b(\OO)$ and $C_b(\XXX)$ of all numerical continuous and bounded functions. The $\Gamma$-convergences are related to these topologies.
\\
It is shown below at Lemma \ref{res-03} that $\Glim k I^k=I,$ meaning that \eqref{MKdyn} is the limit of \eqref{Skdyn}$_{k\ge2}.$

\begin{lemma}\label{res-05}
The function $I$ is coercive.
\end{lemma}

\begin{proof}
As $\left\{P:P_1=\mu_1\right\} $ is closed, it is enough to show that the function $P\mapsto E_P(\ell)+\iota _{\left\{P:P_0=\mu_0\right\} }$ is coercive.
Since $\ell\ge0$ is continuous, $P\mapsto E_P(\ell)
=\sup _{n\ge1}E_P(\ell\wedge n)$ is lower semi-continuous. As in addition $\left\{P:P_0=\mu_0\right\}$ is closed, the function $P\mapsto E_P(\ell)+\iota _{\left\{P:P_0=\mu_0\right\} }$ is also  lower semi-continuous. It remains to show that for every $a\ge0$, $\left\{P:P_0=\mu_0,E_P(\ell)\le a\right\} $ is uniformly tight in $\PO.$ 
\\
For any $n\ge1,$ there is some compact (finite) subset $K_n$ of $\XX$ such that $\mu_0(K_n)\ge 1-1/n.$ We have $\ell\ge N$ where $N$ in the number of jumps, see \eqref{eq-17}. Hence, any $P$ such that  $P_0=\mu_0$ and $E_P(\ell)\le a$ satisfies:
$$
P(X_0\in K_n, N\le n)\ge 1-P(X_0\not\in K_n)-P(\ell> n)
\ge 1-1/n-a/n.
$$
As it is assumed that the graph $(\XX,\sim)$ is locally finite, see \eqref{eq-03}, $\left\{X_0\in K_n, N\le n\right\} $ is a  compact subset of $\OO$ (recall that $\OO$ is compatible with the graph structure, see \eqref{eq-19}). 
This proves the desired uniform tightness and completes the proof of the lemma.
\end{proof}

\begin{lemma}\label{res-06}
For any $P\in\PO,$ there exists a sequence $\seq Pn$ in $\PO$ such that $\Lim n P_n=P,$ $\Lim n E _{P_n}(\ell)=E_P(\ell),$ and $H(P_n|R)<\infty$ for all $n\ge1.$ 
\end{lemma}

A similar result would fail in a diffusion setting with for instance $\OO=C(\ii,\RR)$ and $R$  the reversible Wiener measure (with Lebesgue measure as initial marginal). Here, we are going to take advantage of the countability of the discrete space $\XX$ and of assumption \eqref{eq-01b} which allowed us to introduce  at \eqref{eq-19} a made-to-measure definition of the path space $\OO$. Indeed, this definition is entirely motivated by Lemma \ref{res-06}.

\begin{proof}[Proof of Lemma \ref{res-06}]
Let us pick $P\in\PO.$
\\
(a)\quad The set $\ell(\OO)$ of all possible values of $\ell$ is countable; let us enumerate it: $\ell(\OO)=\left\{c_ n;n\ge 1\right\} $ and expand $P$ along  these values: $P=\sum_n P(\ell=c_n) P(\cdot\mid \ell=c_n).$
\\
(b) \quad Suppose that $Q\in\PO$ is concentrated on the set $\left\{\ell=c\right\} .$ As $\left\{\ell=c\right\}$ is metric separable, there exists a sequence of convex combinations of Dirac masses: $Q_n=\sum _{i=1}^n a^{in} \delta _{\omega^{in}}$ with $\omega ^{in}\in\left\{\ell=c\right\}$ such that $\Lim n Q_n=Q.$
\\
(c) \quad Let $\omega\in\OO$ be a fixed path. We shall prove below that there is a sequence $\seq {Q ^{\omega}} n$ in $\PO$  such that $\Lim n Q_n^\omega=\delta_\omega$ and for each $n$, $Q_n^\omega$ is concentrated on $\left\{\ell=\ell(\omega)\right\} $ and $H(Q_n^\omega|R)<\infty.$

Putting (a), (b) and (c) together, it is not hard to check with the aid of Jensen's inequality applied to the convex function $H(\cdot|R),$ that there exists a sequence $\seq Pn$ in $\PO$ such that $\Lim n P_n=P$ and  for each $n,$ $E _{P_n}(\ell)=E_P(\ell)$ and  $H(P_n|R)<\infty$, which is the desired result.

It remains to prove (c), taking advantage of the specificity of the path space $\OO.$ Let $\omega\in\OO$ be fixed. It is completely described by its jump times $0<t_1<\cdots<t_k<1$ and the corresponding states $(\omega_0,\omega _{t_1},\dots,\omega _{t_k}).$  One  choose $Q^\omega_n\in\PO$ as a Markov probability measure with initial marginal $\delta _{\omega_0}$ and jump kernel  $J_n ^{\omega}=\sum _{i=1}^k \varphi_i^n(t)\,dt\delta _{\omega _{t_i}}$ where the non-negative continuous functions $\varphi_i^n$ have, for each fixed $n\ge 1,$ non-overlapping compact supports as $1\le i\le k$ varies and are such that for each $1\le i\le k,$ $(\varphi_i(t)\,dt)_{n\ge1}$ is an  approximation of $\delta _{t_i}.$
Changing the jump times but keeping the order of $(\omega_0,\omega _{t_1},\dots,\omega _{t_k}),$ doesn't change the value $\ell(\omega).$ Therefore , $Q_n^\omega$ is concentrated on $\left\{\ell=\ell(\omega)\right\} $.
We have $H(Q_n^\omega|R)=E _{Q_n^\omega}\IiX h(\frac{ dJ_n^\omega(t,X_t)}{dJ(t,X_t)}(y))\,J _{t,X_t}(dy) dt $ with $h(a)=a\log a-a +1$ if $a>0$ and $h(0)=1.$ One easily sees that $H(Q^\omega_n|R)<\infty$, using the assumption \eqref{eq-01b} , the fact that  $\omega$ is compatible with the graph structure (by the very definition of $\OO$) and also that $J _{t_i,\omega _{t_i}}(\omega _{t _{i+1}})>0$ for all $i$  (since by Hypothesis \ref{hyp-00}-($R$), $J _{t,x}(y)>0,$ for all $t, x\sim y$).   On the other hand, the compactness of the common initial law $\delta_{\omega_0}$ and the weak convergence of the jump kernels $(J_n^\omega)_{n\ge1}$ to $\sum _{i=1}^k \delta _{\omega _{t_i}}\delta _{t_i}$ which is the jump kernel of $\delta _{\omega}$ implies that $\Lim n Q_n ^{\omega}=\delta _{\omega}$ in $\PO.$ This completes the proofs  of (c) and the lemma.
\end{proof}

\begin{lemma}\label{res-03} The sequence $(I^k)_{k\ge2}$ is equi-coercive and
$\Glim k I^k=I.$
\end{lemma}

\begin{proof}
Let us denote
\begin{equation*}
H^k(P):=H(P|R^k)/\log k,\quad P\in\PO.
\end{equation*}
We first prove the equi-coercivity of $(I^k)_{k\ge2}.$
 Using \eqref{eq-14b}, we obtain
\begin{eqnarray*}
H^k(P)&=&E_P\left[\log(dP/dR)-\log Z^k)\right] /\log k\\
	&=& E_P(\ell)+\left[ H(P|R)-E_P  \int_0^1 J _{t,X _{t}}(\XX)\,dt\right]/\log k +E_P\int _{\iX} k ^{-d(X _{t},y)}\, J _{t,X _{t}}(dy)dt/\log k
\end{eqnarray*}
Because of assumption \eqref{eq-01b},  we have the uniform bounds
\begin{equation}\label{eq-15}
0\le E_P \int_0^1 J _{t,X _{t}}(\XX)\,dt, \ E_P\int _{\iX} k ^{-d(X _{t},y)}\, J _{t,X _{t}}(dy)dt\le \sup _{t,x}J _{t,x}(\XX)<\infty,\forall P\in\PO, k\ge2.
\end{equation}
Hence, we obtain with $I-[(-\inf \eqref{S}\vee 0)+\sup _{t,x}J _{t,x}(\XX)]/\log 2\le I^k$ for all $k\ge2,$ and  Lemma \ref{res-05}  that  $(I^k)_{k\ge2}$ is equi-coercive.
\\
For future use, remark that $H^k(P)<\infty$ if and only if
\begin{equation}\label{eq-20}
 E_P(\ell)<\infty \quad \textrm{and} \quad H(P|R)<\infty.
\end{equation}
Now, we  prove that $\Glim k I^k=I.$
As the constraint set $\left\{P\in\PO; P_0=\mu_0,P_1=\mu_1\right\} $ is closed, it is enough to show that $$\Glim k H^k(P)=E_P(\ell),\quad \forall P\in\PO.$$ 
Since  $P\mapsto E_P(\ell)$ is lower semi-continuous and $H(\cdot|R)\ge0$, with \eqref{eq-15}, we obtain  for any  convergent sequence $P_k\underset{k\to \infty}{\rightarrow} P$ that $\Liminf k H^k(P_k)\ge E_P(\ell)$. Lemma \ref{res-06} tells us that from any recovery sequence $\seq Pn$ for the lower semi-continuity of $P\mapsto E_P(\ell),$ i.e.\ such that $\Lim n E _{P_n}(\ell)=E_P(\ell),$ one can build  a recovery sequence for $(H^k)_{k\ge1},$ i.e.\ $\Lim k H^k(P^k)=E_P(\ell).$ Namely, take $P^k=P _{n(k)}$ with $k\mapsto n(k)$ increasing to infinity slowly enough for $\Lim k H(P _{n(k)}|R)/\log k=0.$ 
This completes the proof the proposition.
\end{proof}

\begin{proposition}\label{res-22}
For any $\mu_0$, $\mu_1\in\PX$, we have
$$
\Lim k \inf \eqref{Sk}=\Lim k \inf \eqref{Skdyn}=\inf \eqref{MKdyn}=\inf \eqref{MK}\in (-\infty,\infty].
$$
\end{proposition}

\begin{proof}
It is a direct corollary of \eqref{eq-13} and Lemma \ref{res-03}.
\end{proof}

The following auxiliary entropic minimization problem will be needed for identifying the limit of $\Ph^k$ as $k$ tends to infinity:
\begin{equation}\label{eq-16}
H(P|R_J)\to \textrm{min};\qquad P\in \mathcal{M}_{\mathrm{dyn}}(\mu_0,\mu_1)
\end{equation}
where  $\mathcal{M}_{\mathrm{dyn}}(\mu_0,\mu_1)\subset \PO$ denotes the set of all minimizers of \eqref{MKdyn} and  
\begin{equation*}
R_J:= \exp \left( \int_0^1 J _{t,X _{t}}(\XX)\,dt\right) \,R\in\MO.
\end{equation*}
Remark that \eqref{eq-01b} ensures the finiteness of the integral in the above exponential.

\begin{lemma}\label{res-04}\    
\begin{enumerate}[(a)]
\item
For each $k\ge2,$ \eqref{Skdyn} has a unique solution $\Ph^k$;
\item
$\mathcal{M}_{\mathrm{dyn}}(\mu_0,\mu_1)$ is a non-empty convex compact subset of $\PO$;
\item
The sequence $(\Ph^k)_{k\ge2}$ is convergent and its limit
$
\Lim k \Ph^k=\Ph\in \mathcal{M}_{\mathrm{dyn}}(\mu_0,\mu_1)
$
is the unique minimizer of \eqref{eq-16}.
\end{enumerate} 
\end{lemma}

Under the assumptions of  Lemma \ref{res-04}, Lemma \ref{res-03} ensures that the limit points of $(\Ph^k)_{k\ge2}$ belong to $\mathcal{M}_{\mathrm{dyn}}(\mu_0,\mu_1).$ But statement (c) of Lemma \ref{res-04} asserts that there is indeed a unique limit point.

\begin{proof}[Proof of Lemma \ref{res-04}] This proof relies on Anzellotti and Baldo's $ \Gamma$-asymptotic expansion technic  \cite{AB93}. For a clear exposition of this technique, see \cite[\S4]{AP03}.

We have seen at \eqref{eq-20}, that $I^k(P)<\infty$ if and only if $E_P(\ell)<\infty,$ $H(P|R)<\infty$ and $P_0=\mu_0, P_1=\mu_1.$ Therefore, taking $P^o$ as in \eqref{eq-00}, we see that $I^k(P^o)<\infty,$ for all $k\ge2.$ Together with the considerations of the preceding section,  this proves statement (a). 
\\
The non-emptiness and convexity parts of statement (b) are immediate. The compactness is a standard consequence of the lower semi-continuity of  $H(\cdot|R),$  the continuity of $P\mapsto P_1$ and the coerciveness of $P\mapsto E_P(\ell)+\iota _{\{P:P_0=\mu_0\}}$, see Lemma \ref{res-05}.
\\
Let us prove (c). Denote $i:=\inf $\eqref{MKdyn}$<\infty$ and consider the subsequent renormalization of $I^k:$
\begin{equation*}
J^k(P):=\log(k) \big(I^k(P)-i\big),\qquad P\in\PO.
\end{equation*}
We have
\begin{equation*}
J^k(P)=\iota _{\left\{P:P_0=\mu_0,P_1=\mu_1\right\}}+\log(k)\big(E_P(\ell)-i\big)+H(P|R_J)
+E_P \IiX k ^{-d(X_t,y)}J _{t,X_t}(dy)\,dt
\end{equation*}
and, using the coerciveness of $H(\cdot|R_J)$ and \eqref{eq-15}, it is easily seen that:
\begin{itemize}
\item
$(J^k)_{k\ge2}$ is equi-coercive;
\item
 $\Glim k J^k=J$ with $J(P)=\iota _{\left\{P:P_0=\mu_0,P_1=\mu_1,E_P(\ell)=i\right\} }+H(P|R_J)$, \ $P\in\PO.$
\end{itemize}
As $H(\cdot|R_J)$ is strictly convex, so is $J$ and \eqref{eq-16} admits a unique minimizer $\Ph$ on the convex set $\mathcal{M}_{\mathrm{dyn}}(\mu_0,\mu_1)=\left\{P:P_0=\mu_0,P_1=\mu_1,E_P(\ell)=i\right\}$. One completes the proof of the lemma, noticing that $\argmin J^k=\argmin I^k=\{\Ph^k\} ,$ for each $k\ge2.$
\end{proof}

\begin{lemma}\label{res-meas}
For each $x,y\in\XX,$ 
$\Gamma ^{xy}$ is measurable. So is $\Gamma.$
\end{lemma}

\begin{proof}
For each $x,y\in\XX,$ denote $\OO_x:=\left\{X_0=x\right\} $ and $\OO ^{xy}:=\left\{X_0=x,X_1=y\right\} .$ The set of $d$-geodesics from $x$ to $y$ is
$ 
\Gamma ^{xy}:=\left\{\omega\in\OO ^{xy}; \ell(\omega)=d(x,y)\right\}.
$ 
Since $\ell$ is continuous and it controls the total number of jumps, the restriction $\ell_x=\ell _{|\OO_x}$ of $\ell$ to the closed set $\OO_x$ is coercive.  Hence $\Gamma ^{xy}=\left\{\omega\in\OO_x;\ell_x=d(x,y)\right\} \cap \left\{X_1=y\right\} $ is a compact subset of $\OO$  (in particular, it is measurable). As a countable union of measurable sets, the set 
$ 
\Gamma:=\cup _{x,y\in\XX} \Gamma ^{xy}
$ 
of all geodesics, is also measurable.
\end{proof}

\begin{lemma}\label{res-07}
The set $\mathcal{M}_{\mathrm{dyn}}(\mu_0,\mu_1)$ consists of all $P\in\PO$ concentrated on $\Gamma,$ i.e.\ $P(\Gamma)=1,$ and such that the  endpoint  marginal $P _{01}\in\PXX$ solves \eqref{MK}.
\end{lemma}

\begin{proof}
Any $P\in\PO$ disintegrates as: $P(\cdot)=\IXX P ^{xy}(\cdot)\,P _{01}(dxdy).$ Thus, $E_P(\ell)=\IXX E _{P ^{xy}}(\ell)\, P _{01}(dxdy).$ As $\ell\ge d(x,y)$ on $\OO ^{xy}$ and $\Gamma ^{xy}=\left\{\ell=d(x,y)\right\} ,$ we have $E_P(\ell)\ge \IXX d(x,y)\, P _{01}(dxdy)$ with equality if and only if $P ^{xy}(\Gamma ^{xy})=1,$ for $P _{01}$-almost every $(x,y).$ This means that $P(\Gamma)=1,$ in which case $E_P(\ell)=\IXX d(x,y)\, P _{01}(dxdy)$ and the conclusion about $P _{01}$ follows immediately. 
\end{proof}

\subsection*{Proof of Theorem \ref{res-0a}-(2-3-4)}
Denote $\Ph\in \PO$ and $\ph\in\PXX$ the unique solutions (if they exist) of \eqref{Stdyn} and \eqref{St}.
\\
We start proving the statements about the dynamical problems \eqref{Skdyn} and \eqref{Sdyn}.
Let $P^o$ be defined by \eqref{eq-00}. Then our assumptions on $\mu_0$ and $\mu_1$ are equivalent to: $P_0^o=\mu_0,$ $P_1^o=\mu_1,$ $E _{P^o}(\ell)<\infty$ and $H(P^o|R)<\infty,$ which are the hypotheses of Lemma \ref{res-04} which tells us that $\Lim k \Ph^k=P^*$ with $P^*$ the unique solution of 
\begin{equation*}
H(P|R_J)\to \mathrm{min};\qquad P\in \mathcal{M}_{\mathrm{dyn}}(\mu_0,\mu_1).
\end{equation*}
Together with $P\in \mathcal{M}_{\mathrm{dyn}}(\mu_0,\mu_1)\iff \left\{\begin{array}{l}
P _{01}\in \SMK(\mu_0,\mu_1)\\
P(\Gamma)=1
\end{array}\right.$ (see Lemma \ref{res-07}),   and the identity $H(P|\Rt)= \left\{\begin{array}{ll}
H(P|R_J),& \textrm{if } P(\Gamma)=1\\
\infty,& \textrm{otherwise} 
\end{array}\right.$, this yields the identity $P^*=\Ph$. 
\\
Formula \eqref{eq-22} with $\ph$ the unique solution of the strictly convex problem \eqref{St}, follows from a reasoning similar to the one leading to \eqref{eq-11} and based on the additive disintegration formula	 \eqref{eq-21}.
\\ 
We prove the statements about the static problems \eqref{Sk} and \eqref{S}  by pushing forward \eqref{Skdyn} and \eqref{Sdyn} with the mapping $(X_0,X_1)$ to obtain the measures $\ph^k=\Ph^k _{01},$ $\ph=\Ph _{01}$ and considering \eqref{eq-21} again.
\hfill$\square$

\section{Dynamics of the displacement random walk}\label{sec-dyn}

To give some detail about the dynamics of the displacement interpolation $\mu$, it is necessary to study the dynamics of the displacement random walk $\Ph.$  It is shown below that for any $x,y\in\XX,$  $\Rt ^{xy}$  and $\Ph$ are Markov and we compute their jump kernels at Theorems \ref{res-0b} and \ref{res-0c}. To achieve this goal, we need some preliminary material involving the reciprocal and Markov properties.

\subsection*{Reciprocal path measure}

The reciprocal property extends the notion of Markov property.
For more detail, see \cite{LRZ12} and the references therein. 
\begin{definition}\label{def-02}
 A measure $Q\in\MO$ is said to be \emph{reciprocal} if for any $0\le u\le v\le 1,$ $Q(X _{[u,v]}\in \cdot\mid X _{[0,u]}, X _{[v,1]})=Q(X _{[u,v]}\in \cdot\mid X_u, X_v).$
\end{definition}
The following lemma is standard. We state it for the comfort of the reader.

\begin{lemma}\label{res-09}\ 
\begin{enumerate}[(a)]
\item
Any Markov measure is reciprocal (but the converse is false).
\item
Almost every  bridge of a reciprocal measure is Markov.
\end{enumerate}
\end{lemma}

\begin{proof} Let us prove (a). Take $Q\in\MO$ any Markov measure. We have for any measurable subsets $A\in \OO _{[u,v]}$ and $B\in\OO _{[v,1]}$ with $Q(X _{[v,1]}\in B)>0,$
\begin{eqnarray*}
Q(X _{[u,v]}\in A\mid X _{[0,u]}, X _{[v,1]}\in B)
&=& \frac{Q(X _{[u,v]}\in A, X _{[v,1]}\in B\mid X _{[0,u]})}{Q(X _{[v,1]}\in B\mid X _{[0,u]})}\\
&=& \frac{Q(X _{[u,v]}\in A, X _{[v,1]}\in B\mid X _u)}{Q(X _{[v,1]}\in B\mid X_u)}\\
&=&Q(X _{[u,v]}\in A\mid X _u, X _{[v,1]}\in B),
\end{eqnarray*}
meaning that $Q(X _{[u,v]}\in \cdot\mid X _{[0,u]}, X _{[v,1]})=Q(X _{[u,v]}\in \cdot\mid X _u, X _{[v,1]}).$ And one concludes using the time-symmetry of the Markov property.

Let us prove (b). Let $Q \in\MO$ be reciprocal. We have for any $x,y\in\XX$ and $0\le t\le1,$
\begin{eqnarray*}
Q ^{xy}(X _{[t,1]}\in \cdot\mid X _{[0,t]})&=&Q(X _{[t,1]}\in \cdot\mid X_0=x, X _{[0,t]}, X_1=y)\\&=&Q(X _{[t,1]}\in \cdot\mid X _t, X_1=y)=Q ^{xy}(X _{[t,1]}\in \cdot\mid X _t)
\end{eqnarray*}
where last equality is obtained repeating the argument with $X_t$ instead of $X _{[0,t]}$.  This is the announced result.
\end{proof}

\begin{lemma}\label{res-10}
Let $Q\in\MO$ be a reciprocal measure and $G\subset \Gamma$ a measurable subset of $\OO$ consisting of geodesics. Then, the measure $Q':= \1 _{G}\,Q\in\MO$ is still reciprocal.
\end{lemma}

\begin{proof}
We use the following property of a geodesic: the restriction $\gamma _{[u,v]}$ of any geodesic $\gamma\in \Gamma,$ is still a geodesic of $\OO _{[u,v]}$. Therefore, $X\in \Gamma ^{X_0,X_1}_{[0,1]}$ implies that $X _{[u,v]}\in \Gamma ^{X_u,X_v}_{[u,v]}$ which implies that 
\begin{eqnarray*}
 &&Q(X\in G,X _{[u,v]}\in A\mid X _{[0,u]}, X _{[v,1]})\\
	&=&\1 _{\{X _{[0,u]}\in G _{[0,u]},X _{[v,1]}\in G _{[v,1]}\}} Q(X _{[u,v]}\in G _{[u,v]}\cap A\mid X _{[0,u]}, X _{[v,1]})\\
	&=&\1 _{\{X _{[0,u]}\in G _{[0,u]},X _{[v,1]}\in G _{[v,1]}\}}
	Q(X _{[u,v]}\in G _{[u,v]} \cap   A\mid X _u, X _v)
\end{eqnarray*} 
for any  measurable set  $A\subset\OO _{[u,v]},$ where the last equality follows from the reciprocal property. Therefore, since $X _{[0,u]}\in G _{[0,u]}$ and $X _{[v,1]}\in G _{[v,1]}$, $Q'\ae,$ we have
\begin{eqnarray*}
Q'(X _{[u,v]}\in A\mid X _{[0,u]},X _{[v,1]})
&=& Q(X _{[u,v]}\in G _{[u,v]}\cap A\mid X _u, X _v)\\
&=&Q'(X _{[u,v]}\in A\mid X _u,X _v),\qquad Q'\ae
\end{eqnarray*}
which is the desired result. 
\end{proof}

\subsection*{Basic properties of $\Rt$.}

We apply Lemmas \ref{res-09} and \ref{res-10} to $\Rt$ defined at \eqref{eq-35}.

\begin{proposition}\label{res-12}
If $R$ is reversible, then $\Rt$ is also reversible.\\
The  measure $\Rt$ is   reciprocal (but not Markov in general)  and it concentrates on the set $\Gamma$ of all geodesics. 
\end{proposition}

\begin{proof}
If $R$ is reversible, the time-reversal invariances of $\int_\ii J _{X_t}(\XX)\,dt$ and $\Gamma$ together with the symmetry of the distance $d$   immediately imply the reversibility of $\Rt.$

Let us show that $R_J:=\exp(\Iii J _{t,X_t}(\XX)\,dt)\, R$ is Markov by proving that for each $t\in\ii$ and all bounded measurable functions $a\in \sigma(X _{[0,t]})$ and $b\in \sigma(X _{[t,1]})$, we have $E _{R_J}(ab\mid X_t)=E _{R_J}(a\mid X_t)E _{R_J}(b\mid X_t).$ Denoting $\alpha:=\exp (\int _{[0,t]}J _{s,X_s}(\XX)\,ds)\in \sigma(X _{[0,t]})$ and $\beta:=\exp (\int _{[t,1]}J _{s,X_s}(\XX)\,ds)\in \sigma(X _{[t,1]}),$ by the Markov property of $R$, we have 
\begin{eqnarray*}
E _{R_J}(ab\mid X_t)=\frac{E_R(a \alpha   b\beta\mid X_t)}{E_R(\alpha   \beta\mid X_t)}=
	\frac{E_R(a \alpha \mid X_t)}{E_R(\alpha   \mid X_t)}
	\frac{E_R(b\beta\mid X_t)}{E_R(   \beta\mid X_t)}
	=E _{R_J}(a\mid X_t)E _{R_J}(b\mid X_t)
\end{eqnarray*}
where last equality is obtained by plugging successively  $b=1$ and $a=1$ in $E _{R_J}(ab\mid X_t)=
	\frac{E_R(a \alpha \mid X_t)}{E_R(\alpha   \mid X_t)}
	\frac{E_R(b\beta\mid X_t)}{E_R(   \beta\mid X_t)}.$ This shows that $R_J$ is Markov.
\\
We conclude with Lemma \ref{res-10} that $\Rt=\1_\Gamma R_J$ is reciprocal.
\end{proof}

Although $\Rt$ is reciprocal, it is not Markov. To see this, remark that the time reversed of  a geodesic is also geodesic. If the geodesic walker only knows that he stands at $z$ at  time $t$, having forgotten his past history and in particular that his previous state before jumping was $z'$, he cannot decide to forbid $z'$ to be his next state. Nevertheless, the bridges $\Rt ^{xy}$ are Markov.

\begin{corollary}\label{res-14z}
For every $(x,y)\in\XXX$, the bridge $\Rt ^{xy}$ is Markov.
\end{corollary}

\begin{proof}
This follows from Lemma \ref{res-09} and Proposition \ref{res-12}, remarking that under our irreducibility assumption, $R _{01}$-almost everywhere is equivalent to everywhere on $\XXX$.						
\end{proof}

As $\Rt ^{xy}$ is Markov, it is sufficient to compute its jump kernel to characterize its dynamics. Recall the definition of the directed tree $(\Gamma ^{xy}(\ii),\to)$, between the statements of Theorems \ref{res-0a} and \ref{res-0b},  that describes the successive occurrence of the states which are visited by the geodesics from $x$ to $y$, regardless of the instants of jump.

\subsection*{Proof of Theorem \ref{res-0b}}
The Markov property is already proved at Corollary \ref{res-14z}.
\\
Let us begin with some notation.
For all $0\le t_1\le t_2\le t_3\le 1,$ $z_1, z_2, z_3\in\XX,$  we denote 
\begin{eqnarray*}
\Gamma(t_1,z_1;t_2,z_2)&:=&\left\{\omega\in \OO; \omega_{|[t_1,t_2]}=\gamma _{|[t_1,t_2]} \textrm{ for some }\gamma\in \Gamma,\omega _{t_1}=z_1,\omega _{t_2}=z_2\right\}\\
\Gamma(t_1,z_1;t_2,z_2;t_3,z_3)&:=&\Gamma(t_1,z_1;t_3,z_3)\cap \left\{X _{t_2}=z_2\right\} .
\end{eqnarray*}
In particular, we have $\Gamma ^{xy}=\Gamma(0,x;1,y).$ 
We also introduce the functions on $\OO:$
\begin{eqnarray*}
G(t_1,z_1;t_2,z_2)&:=&
\exp \left(\int _{t_1}^{t_2}J _{t,X_t}(\XX)\,dt\right) \1 _{\Gamma(t_1,z_1;t_2,z_2)}\\
G(t_1,z_1;t_2,z_2;t_3,z_3)&:=&
\exp \left(\int _{t_1}^{t_3}J _{t,X_t}(\XX)\,dt\right) \1 _{\Gamma(t_1,z_1;t_2,z_2;t_3,z_3)}.
\end{eqnarray*}
We see that \ 
$	
g^y_t(z)= E_R \left(G(t,z;1,y)\mid X_t=z\right) .
$	
\\
As a direct consequence of the definition of a geodesic, for all $0\le t_1\le t_2\le t_3\le 1,$ $z_1\preceq z_2\preceq z_3\in \Gamma ^{xy}(\ii),$  we have
\begin{equation*}
\Gamma(t_1,z_1;t_2,z_2;t_3,z_3)
=\Gamma(t_1,z_1;t_2,z_2)\cap\Gamma(t_2,z_2;t_3,z_3)
\end{equation*}
which implies that
\begin{equation}\label{eq-36}
G(t_1,z_1;t_2,z_2;t_3,z_3)
=G(t_1,z_1;t_2,z_2)G(t_2,z_2;t_3,z_3)\quad \textrm{on }\Gamma ^{xy}.
\end{equation}
As $\Rt ^{xy}$ is Markov, to derive the infinitesimal generator of its Markov semi-group, it is enough to compute its forward stochastic derivative
\begin{equation*}
\Ltxy_tu(z):=\lim _{h\downarrow0} E _{\Rt ^{xy}}\left[u(X _{t+h})-u(X_t)|X_t=z\right],\quad 0\le t<1, z\in \Gamma ^{xy}(\ii).
\end{equation*}
For any $0\le t<1, z\in \Gamma ^{xy}(\ii)$, with \eqref{eq-35} we see that
\begin{eqnarray*}
\Rt ^{xy}(\cdot\mid X_t=z)
&=& \frac{G(0,x;t,z;1,y)}{E_R \left[G(0,x;t,z;1,y)|X_t=z\right] } R(\cdot\mid X_t=z)\\
&=& \frac{G(0,x;t,z)G(t,z;1,y)}{E_R \left[G(0,x;t,z)|X_t=z\right] g^y_t(z)} \, R(\cdot\mid X_t=z)
\end{eqnarray*}
where last equality follows from \eqref{eq-36} and the Markov property of $R.$
\\
We set $U_t=u(X_t)$ for short. For any finitely supported function $u$ and any $0\le t<t+h\le1,$
\begin{eqnarray*}
 &&E _{\Rt ^{xy}}(U _{t+h}-U_t|X_t=z)\\
&=& \frac{E_R \left[(U _{t+h}-U_t)G(0,x;t,z)G(t,z;1,y)|X_t=z\right]}{E_R[G(0,x;t,z)|X_t=z]g^y_t(z)} \\
&\overset{\eqref{eq-36}}=& \frac{1}{g^y_t(z)}E_R \left[(U _{t+h}-U_t)\1 _{\left\{z\le X _{t+h}\le y\right\} }G(t,z;t+h,X _{t+h})
	G(t+h, X _{t+h};1,y)|X_t=z\right] \\
&=& \frac{1}{g^y_t(z)}E_R \left[(U _{t+h}-U_t)\1 _{\left\{z\le X _{t+h}\le y\right\} } G(t,z;t+h,X _{t+h})
	g^y _{t+h}(X _{t+h})|X_t=z\right] 
\end{eqnarray*}
where the Markov property of $R$ is used at last equality.
\\
When $X_t=z$ and $h$ tends down to 0 we have: $$(U _{t+h}-U_t)\1 _{\left\{z\le X _{t+h}\le y\right\} } = \left\{\begin{array}{ll}
0,& \textrm{if }X _{t+h}=X_t=z \textrm{ with probability } 1-J_{t,z}(\XX)h+o(h)\\
u(w)-u(z),& \textrm{if }X _{t+h}=w\leftarrow z \textrm{ with probability } J_{t,z}(w)h+o(h)\\
*,& \textrm{otherwise with  probability } o(h)
\end{array}\right.$$
where $*$ is something bounded by $2\sup|u|,$
and 
$$
G(t,z;t+h,X _{t+h})= \left\{\begin{array}{ll}
1+O(h),& \textrm{if }X _{t+h}=X_t=z\textrm{ with probability } 1-J_{t,z}(\XX)h+o(h)\\
1+O(h),& \textrm{if }X _{t+h}=w\leftarrow z \textrm{ with probability } J_{t,z}(w)h+o(h)\\
*,& \textrm{otherwise with probability } o(h).
\end{array}\right.
$$
where $*$ is something bounded because of the assumption \eqref{eq-01b}.
Hence,
\begin{equation*}
h ^{-1}E _{\Rt ^{xy}}[U _{t+h}-U_t|X_t=z]=
	\sum _{w:z\to w}[u(w)-u(z)] \frac{g^y_t(w)}{g^y_t(z)}J_{t,z}(w)+o _{h\downarrow0}(1)
\end{equation*}
which shows that $\Ltxy u(z)=\sum _{w:z\to w}[u(w)-u(z)] \frac{g^y_t(w)}{g^y_t(z)}J_{t,z}(w)$ and completes the proof of the theorem.
\hfill$\square$

\subsection*{$\Ph$ is Markov}

It follows from \eqref{eq-22}, Proposition \ref{res-12} and \cite[Prop.\,2.8]{LRZ12} that  the limiting path measure $\Ph$ is reciprocal. 
We can do better, but it requires some effort.

\begin{proposition}\label{res-21}
The limiting path measure $\Ph$ is Markov.
\end{proposition}

\begin{proof}
By Theorem \ref{res-0a}-(1), for each $k\ge2,$ 
$\Ph^k$ inherits the Markov property of $R.$ We show below at Lemma \ref{res-20} that, as $k$ tends to infinity, $\Ph^k$ converges in variation norm to $\Ph$ and we conclude with Lemma \ref{res-18} that $\Ph$ is Markov.
\end{proof}

Recall that the total variation norm of the signed bounded measure $q$ on $Y$ is $\|q\|_{\mathrm{TV}}:=|q|(Y)=q^+(Y)+q^-(Y)=\sup _{f:\sup|f|\le 1}\int_Y f\,dq=\sup _{A\subset Y}(|q(A)|+|q(A^c)|)$

\begin{lemma}[Nagasawa, \cite{Naga93}]\label{res-18}
Let $\seq Pk$ be a sequence in $\PO$ of Markov probability measures which converges in variation norm to $P,$ then $P$ is also Markov.
\end{lemma}

\begin{proof}
We have to show  that for all $0\le t\le1$ and any bounded functions $a\in \sigma(X _{[0,t]})$ and $c\in \sigma(X _{[t,1]}),$ we have $E_P(ac)=E_P(a E_P(c|X_t)).$ 
\\
For any bounded functions $b\in\sigma(X_t)$ and  $c\in \sigma(X _{[t,1]})$, we have 
\begin{eqnarray*}
&&\big|E_P[bE_P(c|X_t)]-E_P[bE _{P^k}(c|X_t)]\big|\\
	&\le&
\big|E_P(bc)-E _{P^k}(bc)\big|+\big|E _{P^k}[b E _{P^k}(c|X_t)]-E_P[bE _{P^k}(c|X_t)]\big|\\
&\le&
2\|b\|\|c\| \|P^k-P\|_{\mathrm{TV}} \underset{k\to \infty}\rightarrow 0
\end{eqnarray*}
Hence,
\begin{equation}\label{eq-38}
\Lim k E_{P^k}(c|X_t)=E_P(c|X_t)\quad \textrm{ in } L^1(P).
\end{equation}
Let $a\in \sigma(X _{[0,t]}),$ 
\begin{eqnarray*}
&&
\big|E_P(ac)-E_P[aE_P(c|X_t)]\big|\\
&\le& \big|E_P(ac)-E _{P^k}(ac)\big|+\big|E _{P^k}[aE _{P^k}(c|X_t)]-E_P[aE _{P^k}(c|X_t)]\big|\\
&&\qquad+\big|E_P[aE _{P^k}(c|X_t)]-E_P[aE_P(c|X_t)]\big|
\end{eqnarray*}
But, $$\big|E_P(ac)-E _{P^k}(ac)\big|+\big|E _{P^k}[aE _{P^k}(c|X_t)]-E_P[aE _{P^k}(c|X_t)]\big|
\le 2\|a\|\|c\| \|P^k-P\|_{\mathrm{TV}} \underset{k\to \infty}\rightarrow 0$$
and  $\big|E_P[aE _{P^k}(c|X_t)]-E_P[aE_P(c|X_t)]\big|\underset{k\to \infty}\rightarrow 0$ because of \eqref{eq-38}. This completes the proof of the lemma.
\end{proof}

There are counter-examples of sequences $\seq Pk$ of Markov measures converging narrowly to a  non-Markov $P$.
\\
The following standard lemma (Scheff\'e's theorem) is a preliminary result for Lemma \ref{res-20}'s proof.

\begin{lemma}\label{res-19}
Let $r$ be a positive measure and $p_k=z_k\, r,$ $k\ge1$,  $p=z\,r$ be probability measures which are absolutely continuous with respect to $r.$ If $\Lim k z_k=z,$ $r\ae$, then 
\begin{equation*}
\|z\,r-z_k\,r\|_{\mathrm{TV}}=\int |z-z_k|\,dr \underset{k\to \infty}\rightarrow 0.
\end{equation*}
\end{lemma}

\begin{proof}
Set $d_k=z-z_k.$ Then, $\int d_k\,dr=0.$ For any measurable subset $A$, 
\begin{equation*}
\Big|\int_A d_k\,dr\Big|
+\Big|\int _{A^c} d_k\,dr\Big|\le \int |d_k|\,dr
\end{equation*}
with equality when $A=\left\{d_k\ge0\right\}.$ Therefore, $\|z\,r-z_k\,r\|_{\mathrm{TV}}=\int|d_k|\,dr$. As $\Lim k d_k^+=0,$ $r\ae$ and $0\le d_k^+\le z,$ we obtain $\Lim k\int|d_k|\,dr=2\Lim k\int d_k^+\,dr=0.$
\end{proof}

The following standard lemma will also be used during the proof of Lemma \ref{res-20}.

\begin{lemma}[Laplace principle]\label{res-02}
Let $r$ be a positive measure on the measurable set $Y$. For any measurable function $F:Y\to[-\infty,\infty]$ which is not identically equal to $-\infty$ and any measurable subset $Y'\subset Y$ such that $r(Y')<\infty,$ we have
\begin{equation*}
\lime \epsilon\log \int_{Y'}e ^{F/\epsilon}\,dr=r\esssup _{Y'}F\in (-\infty,\infty].
\end{equation*} 
\end{lemma}

\begin{proof}
Considering the restriction of $r$ to $Y',$
one can assume without loss of generality that $Y'=Y$ and $r(Y)<\infty.$ 
To simplify notation, let us denote $b:=r\esssup F\in (-\infty,\infty]$ and for any $\delta>0,$ $b_\delta:=\min[b-\delta,1/\delta]$.
Since $F\le b,$ $r\ae,$  we have $\epsilon\log\int e ^{F/\epsilon}\, dr\le b+\epsilon\log r(Y)$ and a fortiori $$\limsup _{\epsilon\to0} \epsilon\log\int e ^{F/\epsilon}\,dr\le b.$$
On the other hand, for any $\epsilon,\delta>0,$ 
$$
\epsilon\log\int e ^{F/\epsilon}\,dr\ge \epsilon\log\int \1 _{\left\{F\ge b_\delta\right\}}e ^{F/\epsilon}\,dr\ge \epsilon\log r(F\ge b_\delta)+b_\delta.
$$ 
By the definition of $r\esssup F,$ for any $\delta>0$ we have $0<r(F\ge b_\delta)\le r(Y)<\infty.$ Therefore,  taking the $\liminf _{\epsilon\to0}$ and then letting $\delta$ tend down to zero, we obtain  
$$
\liminf _{\epsilon\to0} \epsilon\log\int e ^{F/a_\epsilon}\,dr\ge b,
$$
which completes the proof of the lemma.
\end{proof}

\begin{lemma}\label{res-20}
We have:\ $\Lim k \|\Ph^k-\Ph\|_{TV}=0.$
\end{lemma}

\begin{proof}
By Lemma \ref{res-19}, it is enough to prove that $\Lim k d\Ph^k/dR=d\Ph/dR,$ $R\ae$
For any $P\in\PO$ such that $P\ll R,$ we have $\frac{dP}{dR}=\sum _{x,y\in\XX} \1 _{\{X_0=x,X_1=y\}}
\frac{P _{01}(x,y)}{R _{01}(x,y)}\ \frac{dP ^{xy}}{dR ^{xy}}.
$
We also have $\Lim k \Ph _{01}^k(x,y)=\Ph _{01}(x,y)$ and $\Lim k \Ph ^{k,xy}=\Ph ^{xy}$  for all $x,y\in\XX.$ Therefore, it remains to show that for each $(x,y)\in\XX,$ 
	 $\Lim k d\Ph ^{xy}/d\Ph ^{k,xy}=1,$ $R ^{xy}\ae$
As, $\Ph ^{xy}=\Rt ^{xy}$ and $\Ph ^{k,xy}=R ^{k,xy}$, this amounts to prove that
\begin{equation}\label{eq-39}
\Lim k \frac{d\Rt ^{xy}}{dR ^{k,xy}}=1,\quad R ^{xy}\ae
\end{equation}
for all $x,y.$
By Girsanov's formula \eqref{eq-14b},  for any $0\le t\le 1,$  $$R^k _{[t,1]}(\cdot|X_t)=k ^{-\ell _{[t,1]}}\exp \Big(\int _{[t,1]\times \XX} (1-k ^{-d(X_s,y)})\,J _{s,X_s}(dy)\,ds\Big)\,R _{[t,1]}(\cdot|X_t)$$ 
where $\ell _{[t,1]}:=\sum _{t\le s \le1}d(X _{s^-},X_s).$ Hence,  the jump measure of $R ^{k,xy}$ is given  for any $t\in\ii,z,w\in\XX$ by
\begin{eqnarray*}
&&J ^{k,xy}_{t,z}(w)\\
&=&
\frac{R^k(X_1=y|X_t=w)}{R^k(X_1=y|X_t=z)}\,k ^{-d(z,w)}\,J_{t,z}(w)\\
&=& \frac{E_R \left[k ^{-\{d(z,w)+\ell _{[t,1]}\}}\exp \left(\int _{[t,1]\times\XX}(1-k ^{-d(X_s,a)})\,J _{s,X_s}(da)ds\right) \1_{(X_1=y)}\mid X_t=w\right]}{E_R \left[k ^{-\ell _{[t,1]}}\exp \left(\int _{[t,1]\times\XX}(1-k ^{-d(X_s,a)})\,J _{s,X_s}(da)ds\right) \1_{(X_1=y)}\mid X_t=z\right]} J_{t,z}(w).
\end{eqnarray*}
With Lemma \ref{res-02}, we obtain that 
\begin{equation*}
\Lim k J ^{k,xy}_{t,z}(w)=
	\frac{E_R \left[\exp \left(\int _t^1J _{s,X_s}(\XX)\,ds\right) \1 _{\Gamma(t,w;1,y)}| X_t=w\right]}{E_R \left[\exp \left(\int _t^1J _{s,X_s}(\XX)\,ds\right) \1 _{\Gamma(t,z;1,y)}| X_t=z\right]}\, J_{t,z}(w)
	=:\Jty_{t,z}(w)
\end{equation*}
meaning, with Theorem \ref{res-0b}, that the pointwise limit of $J ^{k,xy}$ as $k$ tends to infinity is the jump measure $\Jty$ of $\Rt ^{xy}.$ With Girsanov's formula, this implies  that \eqref{eq-39} is true, completing the proof of the lemma.
\end{proof}

\begin{remark}
Theorem \ref{res-0b} and Lemma \ref{res-20}  together, give an alternate proof of  the convergence $\Lim k R ^{k,xy}=\Rt ^{xy}$ stated in Theorem \ref{res-11}, with no reference to an entropy minimization problem.
\end{remark}

\subsubsection*{A comparison with the usual continuous setting}

For comparison, it is worthwhile recalling an analogous result in the usual setting of McCann's displacement interpolations on $\XX=\RR^n$. The Monge-Mather shortening principle, see \cite[Ch.\,8]{Vill09}, tells us that the optimal plan $\ph$ for the quadratic cost between $\mu_0$ and $\mu_1$ is such that two distinct geodesics interpolating between  couples of endpoints in its support  $\supp\ph,$ do not intersect. 
\\
In our graph setting, this would correspond to the nice situation where for any $z$ and $t<1$, $\Ph(X_1\in\cdot|X_t=z)$ reduces to a Dirac measure. But in our discrete setting,  $\Ph(\cdot|X_t=z)$ has sample paths  which live on a directed geodesic tree with top leaves (at time $1$) distributed according to a probability measure $\Ph(X_1\in\cdot|X_t=z)$  which might not reduce to a Dirac mass in the general case.
\\
 This directed tree structure  is a consequence of  the \emph{Markov property} of $\Ph$.
 It can also be seen as a consequence of  the \emph{shortening principle} which, in the present metric cost setting, is an elementary consequence of the triangle inequality.

As a corollary of Proposition \ref{res-21}, we obtain  the following result.

\begin{proposition}
For all $0\le s\le t\le1,$ $\Ph _{st}\in\PXX$ is an optimal coupling of $\mu_s$ and $\mu_t,$ meaning that $\Ph _{st}$ is a solution of \eqref{MK} with $\mu_s$ and $\mu_t$ as  prescribed marginal constraints.
\end{proposition}

\begin{proof}
It is a consequence of 
\begin{enumerate}
\item
the Markov property of $\Ph,$ see Proposition \ref{res-21}, which allows  surgery by gluing the bridges of $\Ph _{[s,t]}$ together with the restrictions $\Ph _{[0,s]}$ and $\Ph _{[t,1]}$;
\item
the fact that $\ell _{st}:=\sum _{s<r<t}d(X _{r^-},X_r)$ is insensitive to  changes of time: i.e.\ for any strictly increasing mapping $\theta:[s,t]\to\ii$ with $\theta(s)=0,$ $\theta(t)=1,$ we have $\ell _{st}=\ell _{01}(X\circ\theta).$
\end{enumerate}
A standard ad absurdum reasoning leads to the announced property.
\end{proof}

\subsection*{Proof of Theorem \ref{res-0c}-(2)}

Now, let us investigate the dynamics of $\Ph$ in the general case where it interpolates between marginal constraints $\mu_0$ and $\mu_1$ which are not necessarily Dirac measures. 
Let us denote $$\widehat p(t,z;dy):=\Ph(X_1\in dy|X_t=z)$$ so that the optimal coupling $\Ph _{t1}$ of $\mu_t$ and $\mu_1$   disintegrates as $\Ph _{t1}(dzdy)=\mu_t(dz)\widehat p(t,z;dy).$
\\
As in Theorem \ref{res-0b}, we compute a stochastic derivative. A general disintegration result tells us that
\begin{eqnarray*}
\Ph _{[t,1]}(\cdot|X_t=z)
	&=&\IX \Ph _{[t,1]}(\cdot| X_t=z, X_1=y)\, \Ph(X_1\in dy|X_t=z)\\
	&=&\IX \Rt _{[t,1]} ^{x_o,y}(\cdot|  X_t=z, )\, \widehat p(t,z;dy)
\end{eqnarray*}
since, for $\mu_0$-almost any $x_o,$
\begin{eqnarray*}
\Ph _{[t,1]}(\cdot| X_t=z, X_1=y)
	&=& \Ph _{[t,1]}(\cdot| X_0=x_o,X_t=z, X_1=y)\\
	&=& \Rt _{[t,1]} ^{x_o,y}(\cdot|  X_t=z)
\end{eqnarray*}
where we used the Markov property of $\Ph$  at the first  equality and the identity $\Ph ^{x_o,y}=\Rt ^{x_o,y}$ at the second equality. Therefore, for all $0<t<t+h<1,$ 
\begin{equation*}
E _{\Ph}\Big([u(X _{t+h})-u(X_t)]/h|X_t=z\Big)
	=\IX E _{\Rt ^{x_o,y}}\Big([u(X _{t+h})-u(X_t)]/h|X_t=z\Big)\, \widehat p(t,z;dy)
\end{equation*}
and letting $h$ tend down to zero, we see that the stochastic derivative of $\Ph$ is given by
\begin{equation*}
\widehat L _{t}u(z)=\IX \Lty _{t}u(z)\, \widehat p(t,z;dy)
\end{equation*}
which gives the desired expression for the jump kernel $\Jh.$
\\
The evolution equation for $[\mu_0,\mu_1]$ is the usual forward Fokker-Planck equation for the time-marginal flow $t\mapsto \Ph_t$ of the Markov measure $\Ph.$
\hfill$\square$

\section{Conservation of the  average rate of mass displacement}\label{sec-conservation}

The main result of this section is Theorem \ref{res-0fb}. It states that the constant average rate of mass displacement, recall Definition \ref{def-05}, is conserved along the  displacement interpolation. It is a consequence of the Corollary \ref{res-39} of Proposition \ref{res-41} which also asserts that some similar quantity  is conserved along the time marginal flow of the solution of the dynamical Schr\"odinger problem \eqref{Sdyn}.

\subsection*{Main results of the section}

The main technical result of this section is the following Proposition \ref{res-41} whose proof is postponed to the next subsection.

\begin{proposition}\label{res-41}
Let $P\in\PO$ be a random walk. We denote $(\nu_t)_{t\in\ii}$ its marginal flow and $(\nu\JJ_t(\XXX))_{0\le t\le1}$ its   average rate of mass displacement. Let us assume that 
\begin{enumerate}[(i)]
\item
$H(P|R)<\infty;$
\item
$\nu\JJ_t(\XXX)>0$ for all $t\in\ii$;
\item
$1<\Iii \nu\JJ_t(\XXX)/\nu J_t(\XXX)\,dt\le \infty.$ 
\end{enumerate}
Then, there exists a change of time $ \widehat\tau$ which minimizes the function $\tau\mapsto H(P^\tau|R)$ among all the changes of time $\tau$ and verifies
\begin{equation}\label{eq-70}
\nu \JJ ^{\widehat \tau}_t(\XXX)-\nu J _{\widehat \tau(t)}(\XXX)=K, \quad \textrm{for almost every }t\in\ii,
\end{equation}
for some constant $K>0.$ We have denoted by $\nu\JJ^{\widehat \tau}_t(\XXX)=\dot{\widehat \tau}(t)\nu\JJ _{\widehat \tau(t)}(\XXX)$ the    average rate of mass displacement of $P ^{\widehat \tau}.$
\end{proposition}

Let us admit this proposition for a while and investigate its consequences.

\newcommand{\Pt}{{\widetilde{P}}}
\newcommand{\jt}{{\widetilde{J}}}

\begin{corollary}\label{res-39}
Let $\Pt$ be the solution of \eqref{Sdyn} and denote $\nu$ its marginal flow and $\jt$ its jump kernel. Suppose that 
$1<\Iii \nu\jt_t(\XXX)/\nu J_t(\XXX)\,dt\le \infty.$ 
Then there exists some $K>0$ such that
$	
\nu\jt_t(\XXX)-\nu J_t(\XXX)=K,\ \textrm{for all }t\in\ii.
$	
\end{corollary}

\begin{proof}
As $\Pt$ solves \eqref{Sdyn}, we have $H(\Pt|R)<\infty$ and also $\nu\jt_t(\XXX)>0$ for all $t\in\ii,$ see \cite[Prop.\,4.2]{Leo12e}. Hence, we are allowed to apply Proposition \ref{res-41}. But as $\Pt$ solves \eqref{Sdyn}, the only  time change $\widetilde\tau$ which minimizes $\tau\mapsto H(\Pt^\tau|R)$ is the identity; recall that \eqref{Sdyn} admits a unique minimizer. Therefore $\widetilde \tau(t)=t$ for all $t\in\ii$ and \eqref{eq-70} becomes $	
\nu\jt_t(\XXX)-\nu J_t(\XXX)=K,$ almost everywhere. Finally we can remove this ``almost" since this limitation comes from \eqref{eq-68} for taking \emph{absolutely continuous} changes of time into account and in the present case $\tau=\mathrm{Id}$ is differentiable everywhere. 
\end{proof}

With Corollary \ref{res-39} at hand, we can prove the main result of the section.

\begin{theorem}\label{res-0fb}
There exists some $K>0$ such that
$\mu\Jh_t(\XXX)=K,$ for all $t\in\ii.$

In particular, when the distance $d$ is the standard discrete distance $d _{\sim}$: i.e.\ for all $x,y\in\XX,$ $x\sim y \iff d _{\sim}(x,y)=1,$ then the displacement interpolation $\mu$ has a constant speed.
\end{theorem}

\begin{proof}
The second statement is a direct consequence of the first one with \eqref{eq-61}.

Let us prove the first statement.
For each $k$, let $\Ph^k$ be the solution of \eqref{Skdyn}.
We denote $\mu^k$ and $\Jh^k$ its marginal flow and jump kernel.
Let us first show that for all $t\in\ii,$
\begin{equation}\label{eq-71}
\Lim k \int _{[0,t]}\mu^k\Jh^k_s(\XXX)\,ds=\int _{[0,t]}\mu\Jh_s(\XXX)\,ds.
\end{equation}
As $\int _{[0,t]}\mu^k\Jh^k_s(\XXX)\,ds=E _{\Ph^k}N_t$ where $N_t:=\sum _{0<s<t}\1 _{\left\{X _{s^-}\not=X_s\right\}}$ is the number of jumps during $[0,t],$ \eqref{eq-71} doesn't directly follow from $\Lim k\Ph^k=\Ph$ (see Theorem \ref{res-0a}), because $N_t$ is unbounded. We must strengthen Theorem \ref{res-0a} to authorize the test functions $0\le N_t\le N_1:=N,$ $0\le t\le 1.$ To do so, it is enough to show that Lemma \ref{res-03} can be improved as follows: the sequence $(I^k)_{k\ge2}$ is equi-coercive and $\Glim k I^k=I$ in $\mathrm{P}_N(\OO):=\left\{P\in\PO; E_PN<\infty\right\}$ with respect to the topology weakened by the functions $C_b(\OO)\cup \left\{N_t;0\le t\le 1\right\} .$
For this strengthening to hold, it is sufficient  that $(I^k)_{k\ge2}$ is equi-coercive in $\mathrm{P}_N(\OO)$.
Inspecting the proof of Lemma \ref{res-03}, we see that $I^k\ge I +c$ for all $k$ and some constant $c$. Consequently, it remains to show that the function $I$ is coercive in $\mathrm{P}_N(\OO)$. But this follows from the proof of Lemma \ref{res-05}, noticing that $0\le N_t\le N\le \ell.$ This completes the proof of \eqref{eq-71}.
\\
We are going to apply Corollary \ref{res-39} to the solution $\Ph^k$ of \eqref{Skdyn}, for any large enough $k\ge1$. Hence, we have to check that $\Iii \mu^k\Jh^k_t(\XXX)/\mu^k J^k_t(\XXX)\,dt>1.$  By \eqref{eq-01b} and since it is assumed that $d(x,y)\ge1$ for all distinct $x,y\in\XX$, 
\begin{equation}\label{eq-72}
\sup _{0\le t\le1}\mu^kJ^k_t(\XXX)\le \bar J/k
\end{equation}
where $\bar J:=\sup _{t,x}J _{t,x}(\XX)<\infty$ 
and with \eqref{eq-71},   
\begin{equation}\label{eq-73}
\Lim k \Iii \mu^k\Jh^k_t(\XXX)\,dt=\Iii \mu\Jh_t(\XXX)\, dt=:K>0.
\end{equation}
It follows that $\Iii \mu^k\Jh^k_t(\XXX)/\mu^k J^k_t(\XXX)\,dt\ge Kk/(2\bar J)>1$, for any large enough $k$.  
\\
Corollary \ref{res-39} tells us that there exists $K_k>0$ such that
$	
\mu^k\Jh^k_t(\XXX)-\mu^k J^k_t(\XXX)=K_k,\ \textrm{for all }t\in\ii.
$ With  \eqref{eq-72} and \eqref{eq-73}, integrating and taking the limit leads us to 
\begin{equation*}
\Lim k \int _{[0,t]}\mu^k\Jh^k_s(\XXX)\,ds=Kt,\quad t\in\ii.
\end{equation*}
We conclude with \eqref{eq-71} that $\int _{[0,t]}\mu\Jh_s(\XXX)\,ds=Kt$ for all $0\le t\le 1,$ which is the announced result.
\end{proof}

\subsection*{Proof of Proposition \ref{res-41}}
A  consequence of Girsanov's formula, see \cite[Thm.\,2.11]{Leo11a} for a related result, is
\begin{equation*}
H(P|R)=H(P_0|m)+\Iii dt \IXX \rho^* \left(\frac{d \JJ _{t,x}}{dJ_{t,x}}(y)\right) \, P_t(dx)J_{t,x}(dy),
\end{equation*}
where 
\begin{equation*}
\rho^*(a):=\left\{\begin{array}{ll}
 a\log a-a+1& \textrm{if } a>0,\\
1& \textrm{if }a=0,\\
\infty& \textrm{if }a<0.
\end{array}\right.
\end{equation*}
Let us denote 
\begin{eqnarray*}
\nu_t(dx)&:=&P_t(dx)\\
\qq_t(dxdy)&:=&\nu_t(dx)J_{t,x}(dy)\\
\vv_t(x,y)&:=& \frac{d \JJ _{t,x}}{dJ _{t,x}}(y)
\end{eqnarray*}
Remark that $\nu=(\nu_t)_{0\le t\le1}, \qq=(\qq_t)_{0\le t\le1}$ and $\vv=(\vv_t)_{0\le t\le1}$ are fixed quantities.
For any positive bounded measure $q\in \mathcal{M}_{b,+}(\XXX)$ and any measurable function $v:\XXX\to\RR$ on $\XXX$, we define
\begin{equation*}
L(q,v):=\IXX \rho^*(v)\,dq\in[0,\infty]
\end{equation*}
For any   change of time $\tau$ we have $P^\tau_0=P_0$  and   
\begin{equation*}
H(P^\tau|R)=H(P_0|m)
+\Iii  L (\qq _{\tau(t)},\dot\tau(t) \vv_{\tau(t)})\, dt.
\end{equation*}
Therefore, we wish to minimize 
\begin{equation}\label{eq-69}
\tau\mapsto \Iii   L (\qq _{\tau(t)},\dot\tau(t) \vv_{\tau(t)})\, dt
\end{equation}
among all    changes of time $\tau:\ii\to\ii.$
To achieve this goal, some preliminary results are necessary. They are stated and proved below at Lemmas \ref{res-38}, \ref{res-36} and \ref{res-37}.
With Lemma \ref{res-38}, we see that we are in position to apply Lemma \ref{res-37}. We conclude with \eqref{eq-69}, Lemma \ref{res-36} and Lemma \ref{res-37}.

It remains to prove Lemmas \ref{res-38}, \ref{res-36} and \ref{res-37}.

\subsubsection*{Statement and  proof of Lemma \ref{res-38}}
We show that the finite entropy condition $H(P|R)<\infty$ implies that the average number of jumps under $P$ is finite.

\begin{lemma}\label{res-38}
Let $P\in\PO$ be a random walk. We denote  $(\nu\JJ_t(\XXX))_{0\le t\le1}$ its   average rate of mass displacement. Let us assume that 
$H(P|R)<\infty.$ Then, $\Iii \nu\JJ_t(\XXX)\,dt<\infty.$
\end{lemma}

\begin{proof}
We see that $$\Iii \nu\JJ_t(\XXX)\,dt=E_P N$$ with $N=\st \1_{\left\{X _{t^-}\not=X_t\right\}}$ the number of jumps.
Let us denote $R_*:=\frac{dP_0}{dm}(X_0)\,R\in\PO$  the Markov random walk with initial marginal $P_0$ and  forward jump kernel $J.$ We have $H(P|R)=H(P_0|m)+H(P|R_*)$ which implies that $H(P|R_*)<\infty.$ Taking advantage of the Fenchel inequality $ab\le a\log a+e ^{b-1},$  we obtain
\begin{equation*}
E_PN=E _{R_*}\left(\frac{dP}{dR_*}N\right) \le H(P|R_*)+E _{R_*}e ^{N}.
\end{equation*}
But our assumption \eqref{eq-01b} implies that $N\pf R_*$ is stochastically dominated by the Poisson law with parameter $\sup _{t,x}J _{t,x}(\XX).$ Therefore, $E _{R_*}e^N<\infty$ and we have proved that $\Iii \nu\JJ_t(\XXX)\,dt<\infty$ when $H(P|R)<\infty.$
\end{proof}

\subsubsection*{Statement and proof of Lemma \ref{res-36}}

To prove Lemma \ref{res-36}, we need  the preliminary Lemma  \ref{res-34} which is stated and proved below.
We consider the  set
\begin{equation*}
\mathcal{K}:=\left\{(q,v); q\in\mathcal{M}_{b,+}(\XXX),v:\XXX\to[0,\infty) \textrm{ measurable }: L(q,v)<\infty,\ q(v>0)>0\right\} .
\end{equation*}
For any $(q,v)\in \mathcal{K},$  as $q$ is a bounded measure,  $v$ belongs to the Orlicz space $L\log L(\XXX,q)$ and its corresponding norm $|v|_q:=\|v\|_{L\log L(\XXX,q)}$ is finite. As in addition  $q(v>0)>0,$ we have $|v|_q>0$ and we are allowed to define
\begin{equation*}
L _{q,v}(\alpha):=L(q,\alpha v/|v|_q),\quad \alpha\in\RR, (q,v)\in \mathcal{K}.
\end{equation*}
Let us fix $K\ge0$ and define $\beta_K(q,v)\in(0,\infty)$ to be the slope of the affine function $\widetilde{L}_{q,v}^K:\RR\to\RR$ which is tangent to the convex function $L _{q,v}$ and satisfies $\widetilde{L}_{q,v}^K(0)=-K:$ $\widetilde{L}_{q,v}^K(\alpha)=\beta _{K}(q,v)\alpha-K,$ $\alpha\in\RR.$ Such a below tangent  exists since $L _{q,v}(0)=L(q,0)=q(\XXX)> 0.$ Furthermore, since $L _{q,v}\ge 0,$ we also have $\beta_K(q,v)>0.$

\vskip 0,5cm
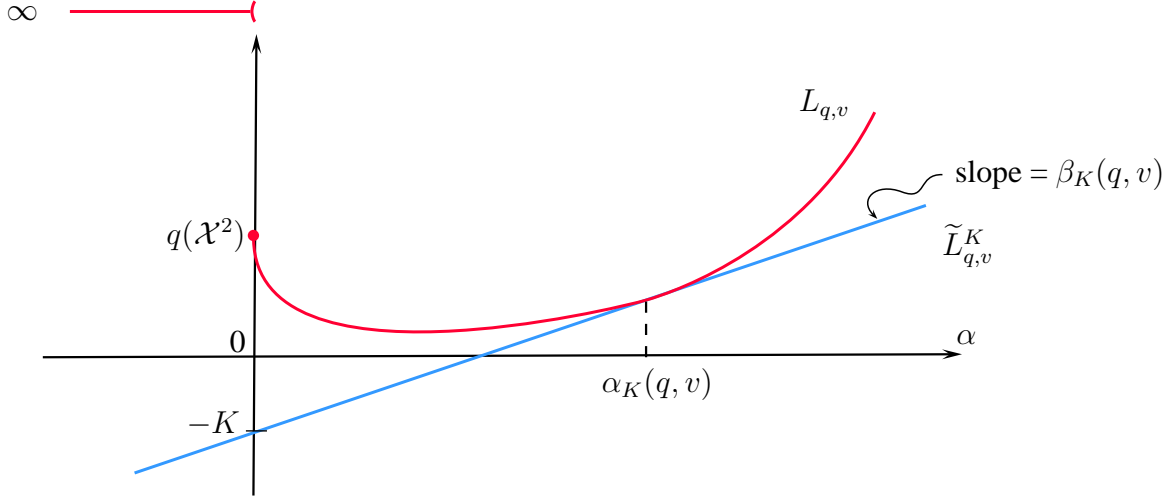
\begin{center}
\scalebox{1} 
{
\begin{pspicture}(0,-3.3041992)(16.04291,3.3291993)
\definecolor{color45}{rgb}{0.2,0.6,1.0}
\definecolor{color46}{rgb}{1.0,0.0,0.2}
\psline[linewidth=0.03cm,arrowsize=0.073cm 2.0,arrowlength=2.0,arrowinset=0.4]{->}(3.6010156,-3.289199)(3.6410155,2.8508008)
\psline[linewidth=0.03cm,arrowsize=0.073cm 2.0,arrowlength=2.0,arrowinset=0.4]{->}(0.8010156,-1.4491992)(13.021015,-1.4091992)
\psline[linewidth=0.04cm,linecolor=color45](2.0210156,-2.9891992)(12.541016,0.5708008)
\psbezier[linewidth=0.04,linecolor=color46,dotsize=0.07055555cm 2.0]{*-}(3.6010156,0.17080078)(3.6610155,-1.9491992)(8.401015,-0.8091992)(8.821015,-0.6891992)(9.241015,-0.5691992)(10.981015,0.05080078)(11.861015,1.8108008)
\psline[linewidth=0.03cm,linestyle=dashed,dash=0.16cm 0.16cm](8.821015,-0.7091992)(8.821015,-1.4491992)
\psline[linewidth=0.04cm,linecolor=color46,tbarsize=0.07055555cm 5.0,rbracketlength=0.15]{)-}(3.6010156,3.1508007)(1.1610156,3.1508007)
\usefont{T1}{ptm}{m}{n}
\rput(13.082471,-1.1841992){$\alpha$}
\usefont{T1}{ptm}{m}{n}
\rput(3.3980956,-1.2441993){0}
\usefont{T1}{ptm}{m}{n}
\rput(3.0724707,-2.3241992){$-K$}
\usefont{T1}{ptm}{m}{n}
\rput(8.97247,-1.8041992){$\alpha_K(q,v)$}
\usefont{T1}{ptm}{m}{n}
\rput(0.5324707,3.1358008){$\infty$}
\usefont{T1}{ptm}{m}{n}
\rput(11.192471,1.9158008){$L_{q,v}$}
\usefont{T1}{ptm}{m}{n}
\rput(13.10247,0.03580078){$\widetilde{L}^K_{q,v}$}
\usefont{T1}{ptm}{m}{n}
\rput(2.9724708,0.17580079){$q(\XXX)$}
\psline[linewidth=0.018cm](3.5010157,-2.4291992)(3.7810156,-2.4291992)
\psbezier[linewidth=0.02,arrowsize=0.05291667cm 2.0,arrowlength=1.4,arrowinset=0.4]{<-}(11.895344,0.39080077)(11.641016,0.65080076)(11.875344,0.8805967)(12.221016,0.7508008)(12.566688,0.6210049)(12.42476,0.9908008)(12.761016,0.97998446)
\usefont{T1}{ptm}{m}{n}
\rput(14.3231735,0.9958008){slope = $\beta_K(q,v)$}
\end{pspicture} 
}
\end{center}

We denote $\alpha _{K}(q,v)>0$ the solution of $L _{q,v}(\alpha)=\widetilde{L}_{q,v}^K(\alpha),$ which happens to be positive and  unique (since $L _{q,v}$ is strictly convex). Define
\begin{equation*}
\left\{\begin{array}{lcl}
\lambda_K(q,v)&:=&\beta_K(q,v)|v|_q,\\
a_K(q,v)&:=& \alpha_K(q,v)/|v]_q,\\
v_K(q,v)&:=&a_K(q,v)v,
\end{array}\right.
\qquad (q,v)\in \mathcal{K}.
\end{equation*}
The functions $\lambda_K(q,\cdot)$, $a_K(q,\cdot)$ and $v_K(q,\cdot)$ are respectively positively homogeneous of degree 1, -1 and 0 on the convex cone $\mathcal{K}_q:=\left\{ v:(q,v)\in\mathcal{K}\right\}.$ We also define 
\begin{equation*}
L_K(q,v):=
\left\{\begin{array}{ll}
\lambda_K(q,v)-K,&\textrm{if }(q,v)\in \mathcal{K},\\
-K,&\textrm{if }q(v\not=0)=0,\\
+\infty,&\textrm{if }q(v<0)>0.
\end{array}\right.
\end{equation*}
Remark that $L_K(q,\cdot)+K$ is the largest positively 1-homogeneous function below the convex function $L(q,\cdot)+K.$

\begin{lemma}\label{res-34} For any $K\ge0, $ we have
$
L_K\le L
$
and for any $(q,v)\in \mathcal{K}$ the  equality $L(q,v)=L_K(q,v) $ is achieved if and only if $v=v_K(q,v)$ or equivalently $a_K(q,v)=1.$
\\
Furthermore, 
for any $(q,v)\in \mathcal{K}$   such that $v>0,$ $q$-a.e., 
$	
\IXX (v_K(q,v)-1)\,dq=K.
$	
\end{lemma}

\begin{proof}
Since
 $L_K(q,v)=\widetilde{L}_{q,v}^K(|v|_q),$ for any $(v,q)\in \mathcal{K},$ the inequality $\widetilde{L}^K _{q,v}\le L _{q,v}$ implies that $L_K\le L$ with equality on $\mathcal{K}$ if and only if $|v|_q=\alpha_K(q,v),$ that is $v=v_K(q,v).$

Let us prove last statement. 
Denoting 
$	
H _{q,v}(\beta):= \sup _{\alpha\in\RR}\left\{\alpha \beta-L _{q,v}(\alpha)\right\} \in\RR,$ $ \beta\in\RR,
$	
the convex conjugate of $L _{q,v},$ we have 
\begin{equation}\label{eq-67}
H _{q,v}(\beta_K(q,v))=K.
\end{equation}
 Let us define, for any measurable function $p:\XXX\to\RR,$
\begin{equation*}
H(q,p):=\IXX \rho(p)\, dq=\IXX (e^p-1)\,dq\in (- \infty,\infty]
\end{equation*}
where
\begin{equation*}
\rho(b):= e^b-1,\quad b\in\RR
\end{equation*}
is the convex conjugate of $\rho^*.$
For any real numbers $a,b:$ $ab\le \rho^*(a)+\rho(b)$ with equality if and only if $a=e^b.$ Hence, for any $(q,v)\in \mathcal{K}$ and any measurable function $p$ on $\XXX$ such that $\IXX e^p\, dq<\infty,$ we have  $-\infty\le \langle p,v\rangle _q:= \IXX pv\,dq\le L(q,v)+H(q,p)<\infty$ with equality if and only if $v=e^p,$ $q$-a.e.
As  $v>0,$ $q$-a.e., the equality is realized with $p=p^v:=\log v,$ that is
\begin{equation}\label{eq-66}
L(q,v)=\langle p^v,v\rangle _q-H(q,p^v).
\end{equation}
Now, let $\alpha:=|v|_q>0$ be given. For any $b\in\RR,$ $\alpha b\le L _{q,v}(\alpha)+H _{q,v}(b)$ and the equality is realized at $\beta=L _{q,v}'(\alpha)=\langle \log v, v/\alpha\rangle _q=\langle p^v,v\rangle_q /\alpha.$ Therefore,
\begin{equation*}
L(q,v)=L _{q,v}(\alpha)=\alpha \beta-H _{q,v}(\beta)=\langle p^v,v\rangle _q-H _{q,v}(\beta).
\end{equation*}
Comparing with \eqref{eq-66} leads us to 
$
H(q,p^v)=H _{q,v}(\beta).
$
In particular, with $v=v_K(q,v),$ we see that $\alpha=|v_K(q,v)|_q=\alpha_K(q,v)$ and the corresponding conjugate parameter is $\beta=\beta_K(q,v).$ It follows with \eqref{eq-67} that $H(q,\log v_K(q,v))=H(q,p ^{v_K(q,v)})=H _{q,v}(\beta_K(q,v))=K,$ which is the desired result.
\end{proof}

As a consequence of Lemma \ref{res-34}, we obtain the following result.

\begin{lemma}\label{res-36}
Suppose that $\widehat \tau$ is a change of time which solves  the differential equation
\begin{equation}\label{eq-68}
\dot \tau(t)=a_K(\qq _{\tau(t)},\vv _{\tau(t)}), \quad t\in\ii, \textrm{a.e.}
\end{equation}
for some $K\ge0.$ Then, $\tau\mapsto \Iii L(\qq _{\tau(t)},\dot \tau(t)\vv _{\tau(t)})\,dt $ attains its minimum value among all  changes of time at $\tau=\widehat\tau.$ Moreover, if for almost all $t\in\ii$, $\vv_t>0,$ then 
\begin{equation*}
\IXX (\widehat\vv_t-1)\,d\widehat\qq_t=K,\quad \textrm{for a.e. }t\in\ii
\end{equation*}
where we denote $\widehat\qq_t:=\qq _{\widehat\tau(t)}$ and $\widehat \vv_t:=\dot {\widehat\tau}(t)\vv _{\widehat\tau(t)}$.
\end{lemma}

\begin{proof}
As $\widehat \tau$ solves \eqref{eq-68}, we have $a_K(\widehat\qq_t,\widehat\vv_t)=1,$ for a.e. $t\in\ii.$ With Lemma \ref{res-34}, this implies that 
\begin{equation*}
\Iii L(\widehat\qq_t,\widehat\vv_t)\,dt=\Iii L_K(\widehat\qq_t,\widehat\vv_t)\,dt.
\end{equation*}
As $\lambda_K(q,\cdot)$ is 1-homogeneous, we see that for any  change of time $\tau,$
\begin{equation*}
\begin{split}
\Iii L_K(\qq _{\tau(t)},\dot\tau(t)\vv _{\tau(t)})\,dt
	&=\Iii \lambda_K(\qq _{\tau(t)},\dot\tau(t)\vv _{\tau(t)})\,dt-K
	\\&=\Iii \lambda_K(\qq_t,\vv_t)\,dt-K
	=\Iii L_K(\qq_t,\vv_t)\,dt
\end{split}
\end{equation*}
is a quantity which doesn't depend on $\tau.$ Hence, for any  change of time $\tau,$
\begin{equation*}
\begin{split}
\Iii L(\widehat\qq_t,\widehat\vv_t)\,dt=\Iii L_K(\widehat\qq_t,\widehat\vv_t)\,dt
&=\Iii L_K(\qq _{\tau(t)},\dot\tau(t)\vv _{\tau(t)})\,dt
	\\&\le \Iii L(\qq _{\tau(t)},\dot\tau(t)\vv _{\tau(t)})\,dt
\end{split}
\end{equation*}
where the last inequality follows from $L_K\le L,$ see Lemma \ref{res-34}.
This proves the minimal attainment at $\widehat \tau.$
\\
The last statement follows from $\eqref{eq-68}\iff \widehat\vv_t=v_K(\widehat\qq_t,\widehat\vv_t)$ and Lemma \ref{res-34}.
\end{proof}

\subsubsection*{Statement and proof of Lemma \ref{res-37}}

Now, we prove Lemma  \ref{res-37}. To do so,  we need the following preliminary Lemma \ref{res-35}.

\begin{lemma}\label{res-35}
For each $K\ge0$, 
\begin{equation*}
a_K(q,v)=\frac{K+q(\XXX)}{\IXX v\,dq},\quad (q,v)\in \mathcal{K}.
\end{equation*}
\end{lemma}

\begin{proof}
Let us pick $(q,v)\in\mathcal{K}$ and $\beta\in\RR.$ We have $H _{q,v}(\beta)=\sup _{a\in\RR}\left\{a \beta-L _{q,v}(a)\right\}$ and the supremum is attained at $\alpha$, solution of $\beta=\IXX \rho ^{*\prime}(\alpha u)u\, dq=\IXX \log(\alpha u)u\,dq$ where $u:=v/|v|_q.$  Therefore, $H _{q,v}(\beta)=\alpha \beta- L _{q,v}(\alpha)=\IXX [\alpha u\log(\alpha u)-\rho^*(\alpha u)]\,dq=\IXX (\alpha u-1)\,dq.$
Choosing $\beta=\beta_K(q,v)$ corresponds to $\alpha_K(q,v)$ and we obtain with \eqref{eq-67} that $K=H _{q,v}(\beta_K(q,v))=\alpha_K(q,v)\IXX u\,dq-q(\XXX).$ Hence, $a_K(q,v)=\alpha_K(q,v)/|v|_q=(K+q(\XXX))/\IXX v\,dq.$
\end{proof}

As a consequence of Lemma \ref{res-35}, we obtain the following result.

\begin{lemma}\label{res-37}
Let $\nu$ and $j$ be such that 
\begin{enumerate}[(i)]
\item
$\Iii \nu j_t(\XXX)\,dt<\infty;$
\item
$\nu j_t(\XXX)>0,$ for all $t\in\ii;$
\item
$1<\Iii \nu j_t(\XXX)/\nu J_t(\XXX)\, dt\le \infty.$
\end{enumerate}
Then, there exist a constant  $\widehat K>0$ and a change of time $\widehat \tau$ such that \eqref{eq-68} holds: 
\begin{equation*}
\dot{\widehat \tau}(t)=a _{\widehat K}(\qq _{\widehat\tau(t)},\vv _{\widehat\tau(t)}), \quad t\in\ii, \textrm{a.e.}
\end{equation*}
\end{lemma}

\begin{proof}
Plugging $\qq_s:=\nu J_s$ and $\vv_s(x,y):=d \JJ _{s,x}/dJ _{s,x}(y)$ into $a_K(\qq_s,\vv_s),$ we see that (ii) implies that $(\qq_s,\vv_s)\in \mathcal{K}$ for all $0\le s\le 1.$
Hence, with Lemma \ref{res-35} we obtain that \eqref{eq-68} writes as
\begin{equation*}
\dot \tau(t) \varphi_K(\tau(t))=1
\end{equation*} 
where
\begin{equation*}
\varphi_K(s):=a_K(\qq_s,\vv_s)^{-1}=\frac{\nu \JJ_s(\XXX)}{K+\nu J_s(\XXX)}, \quad 0\le s\le 1.
\end{equation*}
Let $\Phi_K(t):=\int_0^t \varphi_K(s)\,ds\in [0,\infty].$ By assumption (i), for any $K>0,$  $\Phi_K(1)<\infty.$ It follows that $\Phi_K$ is absolutely continuous. Assumption (ii) implies that it is  strictly increasing so that one can define $\tau_K(t):=\Phi_K ^{-1}(t)$ which   solves $\dot \tau(t)=a_K(\qq _{\tau(t)},\vv _{\tau(t)})$  almost everywhere on $[0,\Phi_K(1))$.
\\
It remains to show that there exists some $\widehat K>0$ such that $\Phi _{\widehat K}(1)=1.$ But this is insured by the assumption (iii) which states that $\Phi_0(1)>1,$ since with (i) we see that $K\in(0,\infty)\mapsto \Phi_K(1)\in(0,\infty)$ is a  continuous decreasing function from $\Phi_0(1)$ to $\Lim K \Phi_K(1)=0$. 
\end{proof}

\appendix

\section{Random walk on a graph}\label{sec-RW}

We give some basic informations about random walks on a graph. 

\subsection*{The jump kernel}

A Markov random walk on the graph $(\XX,\sim~)$ is a time-continuous Markov process which is specified by its infinitesimal generator $L=(L_t)_{0\le t<1}$ acting on any real function $u\in\RR^\XX$ with a finite support via the formula
\begin{equation*}
L_t u(x)=\sy[u(y)-u(x)]\,\JJ_{t,x}(y),
\qquad  x\in\XX,\ 0\le t <1
\end{equation*}
where $\JJ_{t,x}(y)\ge0$   is the average frequency of jumps from $x$ to $y$ at time $t$. 
As a convention, we set $\JJ_{t,x}(y)=0$ as soon as $x$ and $y$ are not neighbours. The corresponding jump kernel is
$$
\JJ _{t,x}:=\sy \JJ _{t,x}(y)\delta_y\in\MX,\qquad x\in\XX,\ 0\le t<1.
$$  
The operator $L$ is well-defined for any bounded function $u\in\RR^\XX$ provided that 
\begin{equation}\label{eq-01}
\sup _{x\in\XX}\JJ _{t,x}(\XX)<\infty, 
\end{equation}
where  $\JJ _{t,x}(\XX):=\sy \JJ _{t,x}(y)$ denotes  the global jump intensity  at $x\in\XX$. The bound \eqref{eq-01} ensures that the random walk performs almost surely  finitely many jumps during the unit time interval $[0,1].$ Therefore, $\OO$ as defined in this article is the relevant path space to be considered.
When $\JJ$ is assumed to satisfy \eqref{eq-01}, it uniquely specifies $Q\in\MO$ up to its initial law $Q_0\in\MX$. 
\\
In case $L$ doesn't depend on $t$, the random walk is said to be time-homogeneous.
In this special case, its dynamics is described as follows. Once  at site $x,$ the walker waits during a random time with exponential law  $\mathcal{E}(\JJ_x(\XX))$, and then decides to jump at $y$ according to the probability $\JJ_x(\XX) ^{-1}\sy \JJ_x(y) \delta_y,$ and so on; all these random events being mutually independent.

\subsection*{Some interesting examples of reference random walks $R$}

One may require  that $R$ is reversible. This means that it is time-homogeneous and that there is a (possibly unbounded) positive measure $m\in\MX$ on $\XX$ such that, not only $R$ is $m$-stationary i.e.: 
$R_t=m, \forall 0\le t\le1,$ but also that $R$ is invariant with respect to time reversal i.e.: for any subinterval $[u,v]\subset\ii,$ $$(X _{(u+v-t)^-};u\le t\le v)\pf R=(X_t;u\le t\le v)\pf R.$$
This happens if and only if the following time-homogeneous detailed balance condition
\begin{equation}\label{eq-02}
m_x J_x(y)=m_y J_y(x),\quad \forall x\sim y\in\XX
\end{equation}
is satisfied; compare \eqref{eq-94}. As it is assumed that 
$ 
J_x(y)>0, \forall x\sim y
$ 
 and the graph is irreducible, this implies that $m_x>0$ for all $x\in\XX.$ 

\subsubsection*{Simple random walk}

An important example of such a walk is the  simple random walk $R^o\in\PO$ on $\XX.$ 
The dynamics of $R^o$  is specified by the jump kernel
\begin{equation}\label{eq-98a}
J^o_x:= n_x ^{-1}\sy \delta_y,\qquad x\in\XX.
\end{equation}
The successive waiting times are independent and identically distributed with the exponential law $\mathcal{E}(1)$ and the walker  jumps from any site choosing a neighbour uniformly at random. Solving \eqref{eq-02}, one sees that the corresponding reversing measures are multiples of
\begin{equation}\label{eq-98b}
m^o:=\sx n_x \delta_x.
\end{equation}
Note that $m^o$ is unbounded whenever $\XX$ is an infinite set, since the irreducibility assumption implies that $n_x\ge 1$ for all $x.$\\ As the simple random walk is analogous to the Brownian motion on a Riemannian manifold, the measure $m^o$ plays the role of the volume measure on the graph.

\subsubsection*{Counting random walk}
It corresponds to 
$ 
J_x=\sy \delta_y, x\in\XX
$ 
whose reversing measure is the counting measure $m=\sum _{x\in\XX}\delta_x.$

\subsubsection*{A generic class of $m$-reversible random walks}

Take some measure $m=\sx m_x \delta_x$ on $\XX$ with $m_x>0, \forall x\in \XX$  and consider the jump kernel
\begin{equation}\label{eq-44}
J_x:=\sy \frac{s(x,y)}{\sqrt{n_xn_y}}\sqrt{\frac{m_y}{m_x}}\,\delta_y,\qquad x\in\XX,
\end{equation}
where $s$ is a symmetric function.
Assume that
 there exists some constant $1\le c<\infty$ such that
\begin{equation*}
m_y/n_y\le c\ m_x/n_x,\quad \forall x\sim y
\end{equation*}
and some $0< \sigma<\infty$ such that the symmetric function $s$ satisfies
\begin{equation*}
0<s(x,y)=s(y,x)\le \sigma,\qquad \forall x\sim y\in\XX.
\end{equation*}
Then, $J$  verifies \eqref{eq-01}, so that $R$ is a measure on $\OO.$
As  $J$ clearly verifies \eqref{eq-02},  $R$ is $m$-reversible. Moreover,  $R$ is equivalent to the simple random walk $R^o$, i.e.\ for any measurable $A\subset\OO,$ $R(A)>0\iff R^o(A)>0$.

\section{Relative entropy}\label{sec-ent}

This  section is a short version of \cite[\S\,2]{Leo12b} which we refer to for more detail.
Let $r$ be some $\sigma$-finite positive measure on some  space $Y$. The relative entropy of the probability measure $p$ with respect to $r$ is loosely defined by
\begin{equation*}
H(p|r):=\int_Y \log(dp/dr)\, dp\in (-\infty,\infty],\qquad p\in \PY
\end{equation*}
if $p\ll r$ and $H(p|r)=\infty$ otherwise. 
\\
More precisely, when $r$ is a probability measure,  we have $$H(p|r)=\int_Y h(dp/dr)\,dr\in[0,\infty],\qquad p,r\in\PY$$ with $h(a)=a\log a-a+1\ge 0$ for all $a\ge0,$ (set $h(0)=1).$ Hence,  the above definition  is meaningful. It follows from the strict convexity of $h$ that $H(\cdot|r)$ is also strictly convex. In addition, since $h(a)=\inf h=0 \iff a=1,$ we also have for any $ p\in \PY,$
\begin{equation}\label{eq-08}
H( p|r)=\inf H(\cdot|r)=0\iff  p=r.
\end{equation}
If $r$ is unbounded, one must restrict the definition of $H(\cdot|r)$ to some subset of $\PY$ as follows. As $r$ is assumed to be $\sigma$-finite, there exists a  measurable function $W:Y\to [1,\infty)$ such that
\begin{equation}\label{eq-09}
z_W:=\int_Y e ^{-W}\, dr<\infty.
\end{equation}
Define the probability measure $r_W:= z_W ^{-1}e ^{-W}\,r$ so that $\log(dp/dr)=\log(dp/dr_W)-W-\log z_W.$ It follows that for any $p\in \PY$ satisfying $\int_Y W\, dp<\infty,$ the formula 
\begin{equation}\label{eq-10}
H(p|r):=H(p|r_W)-\int_Y W\,dp-\log z_W\in (-\infty,\infty]
\end{equation}
is a meaningful definition of the relative entropy which is coherent in the following sense. If $\int_Y W'\,dp<\infty$ for another measurable function $W':Y\to[0,\infty)$ such that $z_{W'}<\infty,$ then $H(p|r_W)-\int_Y W\,dp-\log z_W=H(p|r _{W'})-\int_Y W'\,dp-\log z_{W'}\in (-\infty,\infty]$.
\\
Therefore, $H(p|r)$ is well-defined for any $p\in \PY$ such that $\int_Y W\,dp<\infty$ for some measurable nonnegative function $W$ verifying \eqref{eq-09}. 
\\
It follows from the strict convexity of $H(\cdot|r_W)$ and \eqref{eq-10} that $H(\cdot|r)$ is also strictly convex.

Let $Y$ and $Z$ be two  Polish spaces equipped with their Borel $\sigma$-fields. For any measurable function $\phi:Y\to Z$ and any measure $q\in \mathrm{M}_+(Y)$ we have the  disintegration formula
\begin{equation*}
q(dy)=\int _{Z} q(dy|\phi=z)\, \phi\pf q(dz)
\end{equation*}
where $z\in Z\mapsto q(\cdot|\phi=z)\in \mathrm{P}(Y)$ is measurable, and the following additivity property
\begin{equation}\label{eq-33}
H(p|r)=H(\phi\pf p|\phi\pf r)+\int _{Z} H\Big(p(\cdot\mid \phi=z)|r(\cdot\mid\phi=z)\Big)\,\phi\pf p(dz),
\end{equation}
 is valid for any $p\in \PY$ and any $\sigma$-finite $r\in \mathrm{M}_+(Y).$
In particular, as $r(\cdot\mid\phi=z)$ is a probability measure for each $z$, with \eqref{eq-08} we see that
\begin{equation*}
H(\phi\pf p|\phi\pf r)\le H(p|r),\quad \forall p\in\PY
\end{equation*}
with equality if and only if 
\begin{equation}\label{eq-34}
p(\cdot\mid \phi=z)=r(\cdot\mid\phi=z),\quad \forall z, \phi\pf p\as
\end{equation}


\end{document}